\documentclass[preprints,article,accept,pdftex,moreauthors]{Definitions/mdpi} 
\usepackage[english]{babel}

\usepackage[utf8]{inputenc}
\usepackage{ucs}
\usepackage[T1]{fontenc}
\usepackage{amsmath} 
\usepackage{amssymb} 
\usepackage{amstext}
\usepackage{amsthm}
\usepackage[normalem]{ulem}
\usepackage{color}
\usepackage{xcolor}
\usepackage{extarrows}
\usepackage{stmaryrd}
\usepackage{framed}
\usepackage{enumitem}

\usepackage{graphicx}
\usepackage{cancel}
\usepackage{dsfont}
\usepackage{mathtools}

\usepackage{bbm} 

\usepackage{comment}

\usepackage{listings}

\usepackage{placeins}

\usepackage{setspace}

\usetikzlibrary{arrows,shapes,trees,patterns}
\newcommand{\abs}[1]{\left\vert#1\right\vert}

\renewcommand{\tilde}{\widetilde}

\allowdisplaybreaks

\newcommand*\diff{\mathop{}\!\mathrm{d}}

\theoremstyle{plain}
\newtheorem{lemma}[theorem]{Lemma}
\newtheorem{proposition}[theorem]{Proposition}
\newtheorem{corollary}[theorem]{Corollary}

\theoremstyle{definition}
\newtheorem{example}[theorem]{Example}
\newtheorem{definition}[theorem]{Definition}

\newtheorem{remark}[theorem]{Remark}

\theoremstyle{plain}
\newtheorem*{theorem*}{Theorem}
\newtheorem*{lemma*}{Lemma}
\newtheorem*{proposition*}{Proposition}
\newtheorem*{corollary*}{Corollary}

\theoremstyle{definition}
\newtheorem*{beispiel*}{Beispiel}
\newtheorem*{definition*}{Definition}

\theoremstyle{remark}
\newtheorem*{bemerkung*}{Bemerkung}

\firstpage{1} 
\pubvolume{1}
\issuenum{1}
\articlenumber{0}
\pubyear{2022}
\copyrightyear{2022}
\datereceived{} 
\dateaccepted{} 
\datepublished{} 
\hreflink{https://doi.org/} 

\Title{Optimal dividends for a two-dimensional risk model with simultaneous ruin of both branches}

\TitleCitation{Optimal dividends for a two-dimensional risk model: Simultaneous ruin of both branches}

\Author{Philipp Lukas Strietzel $^{1,\ddagger,}$*\orcidA{}, Henriette Elisabeth Heinrich $^{2,\ddagger}$\orcidB{}}

\AuthorNames{Philipp Lukas Strietzel, Henriette Elisabeth Heinrich}

\AuthorCitation{Strietzel, P. L.; Heinrich, H. E.}

\address[1]{%
$^{1}$ \quad Technische Universit\"at Dresden, Institut f\"ur Mathematische Stochastik, 01062 Dresden, Germany; \url{philipp.strietzel@tu-dresden.de}\\
$^{2}$ \quad Technische Universit\"at Dresden, Institut f\"ur Mathematische Stochastik, 01062 Dresden, Germany; \url{henriette.heinrich@tu-dresden.de}}

\corres{Correspondence: \url{philipp.strietzel@tu-dresden.de}} 

\secondnote{These authors contributed equally to this work.}

\abstract{
		We consider the optimal dividend problem in the so-called degenerate bivariate risk model under the assumption that the surplus of one branch may become negative. More specific, we solve the stochastic control problem of maximizing discounted dividends until simultaneous ruin of both branches of an insurance company by showing that the optimal value function satisfies a certain Hamilton-Jacobi-Bellman (HJB) equation. Further, we prove that the optimal value function is the smallest viscosity solution of said HJB equation, satisfying certain growth conditions. Under some additional assumptions, we show that the optimal strategy lies within a certain subclass of all admissible strategies and reduce the two-dimensional control problem to a one-dimensional one. The results are illustrated by a numerical example and Monte-Carlo simulated value functions. 
}

\begin{document}

\section{Introduction} \label{S1}

	The problem of paying dividends optimally naturally arises when considering insurance risk processes. This is due to the fact that for any classical \emph{Cramér-Lundberg} or even any spectrally negative \emph{Lévy} risk model, the process either drifts to infinity with positive probability, or it faces ruin almost surely. As the assumption that an insurance company's surplus grows to infinity is unrealistic, dividend payments are the natural choice to avoid this behavior. 
	
	In the univariate setting, optimal dividend payments are a well-studied field that was first introduced by De Finetti in \cite{DeFinetti1957} who considered the dividend barrier model. 
	Later, in \cite{Gerber1969}, it was shown that in the Cramér-Lundberg risk model, the optimal strategy is always a band strategy. The study of dividends in this classical risk model has been continued with several extensions, see e.g. the works \cite{Azcue2005,Azcue2010,Azcue2012} and \cite{Azcue2014}. The former three publications treat the cases of optimal dividend strategies with reinsurance, with investment in a Black-Scholes market, and the case of bounded dividend rates, respectively. The textbook \cite{Azcue2014} gives a broad overview of how to use the theory of stochastic control and the dynamic programming approach to tackle dividend problems. In \cite{Thonhauser2007} this approach is used to solve the problem of optimal dividend payments in the presence of a reward for later ruin. Also in the field of (spectrally negative) Lévy risk models, there are several works regarding the dividend problem. See e.g. \cite{Avram2007}, where the optimal dividend policy for a spectrally negative Lévy process with and without bailout loans is considered and the optimal strategy among all barrier strategies is identified, or \cite{Loeffen2008}, where these results are extended by giving sufficient conditions on when the optimal strategy is indeed of barrier type. For a more thorough overview of the available tools, approaches, optimality results and optimal strategies in the univariate setting we refer to \cite{Albrecher2009,Avanzi2009} and \cite{Schmidli2008}.  
	
	In multivariate risk theory, a popular model is the so-called \emph{degenerate bivariate risk model}: Given a Poisson process $N=\left(N(t)\right)_{t \ge 0}$ with rate $\lambda>0$, define a claim process by 
	
	\begin{equation} \label{eq_DefinitionClaimProcess}
		S(t) = \sum_{n=1}^{N(t)} U_n,	\qquad t\geq 0,
	\end{equation} 
	
	where $U_n$ are non-negative i.i.d. claim sizes with cdf $F$, independent of $N$. The degenerate model is then given via 
	
	\begin{equation} \label{eq_DefinitionModel}
		\mathbf{X}(t) = \begin{pmatrix} X_1(t) \\ X_2(t)\end{pmatrix} :=  \begin{pmatrix}
			x_1 + c_1 t - \sum_{n=1}^{N(t)}b_1 U_n \\
			x_2 + c_2 t - \sum_{n=1}^{N(t)}b_2 U_n
		\end{pmatrix} =: \mathbf{x} + \mathbf{c}t - \mathbf{b}S(t),
	\end{equation}
	
	where $x_1,x_2\geq 0$ are the initial capitals of the two branches of the insurer, $c_1,c_2>0$ define constant premium rates, and $b_1,b_2>0$ define the proportion of each claim covered by the corresponding branch. As the claims need to be fully covered, we assume $b_1+b_2 = 1$. \\
	The degenerate model can be seen as an insurer-reinsurer model with proportional reinsurance, where the insurer covers proportion $b_1$ of each claim, while the reinsurer covers proportion $b_2$. The model was first introduced in \cite{Avram2008}, where the authors derive the Laplace transform of the probability of ruin of at least one branch. Further on, the study of ruin probabilities in the degenerate model under different assumptions has been continued, see e.g. \cite{Avram2009,Avram2008,Badescu2011, Foss2017, Hu2013}. Also, in the field of optimal dividends, the degenerate model has gained attention, e.g. in \cite{Czarna2011} where dividends are paid according to an impulse or refraction control and ruin corresponds to exiting the positive quadrant. Under the same ruin assumption, in \cite{Azcue2018}, it was shown that the optimal value function for general admissible strategies satisfies a certain Hamilton-Jacobi-Bellmann equation and the optimal strategy is described. \\
	In the present work, we adapt the stochastic control approach used in \cite{Azcue2018} and consider the problem of paying dividends optimally under the assumption that one branch of the insurance company may have a negative surplus. I.e., we extend the set, where the insurance company is considered \emph{solvent} to $\mathbb{R}^2\backslash\mathbb{R}^2_{<0}$. We show that the optimal value function, which represents the expected discounted value of the paid dividends of the optimal strategy, may be characterized as the smallest viscosity solution of a certain Hamilton-Jacobi-Bellman (HJB) equation. Further, we show that in certain subcases, the optimal strategy lies within the subclass of so-called \emph{bang} strategies, which reduces the bivariate optimization problem to a univariate one. \\
	The paper is organized as follows. In Section \ref{Section_Preliminaries} we specify our model and introduce necessary notations. Afterward, in Section \ref{Section_OptimalValueFunction} we derive the HJB equation and show that it is satisfied by the optimal value function. Section \ref{Section_CandidateOptimalStrategy} is dedicated to the bang strategies, whereas in Section \ref{Section_ApproximationSimulation} we use Monte-Carlo simulation to approximate the optimal strategy for a certain example.
	
	\section{Preliminaries} \label{Section_Preliminaries}
	
	Throughout the whole paper we are going to use the bold version of any variable for the two-dimensional version of the same letter (e.g. $\mathbf{x} = (x_1,x_2)$) without explicitly stating it again. All random quantities are defined on a filtered probability space $(\Omega,\mathcal{A},\mathbb{F},\mathbb{P})$ and $\mathbb{P}_{\mathbf{x}},~   \mathbb{E}_{\mathbf{x}}$ denote the probability measure and expectation, given that $\mathbf{X}(0)=\mathbf{x}$, respectively. The derivative of a function $u$ with respect to the variable $x_i$ is denoted by $\partial_{x_i}u(\mathbf{x})$ or $u_{x_i}(\mathbf{x})$. For any set $M$ we denote by $M^\circ$ the inner of $M$. \\
	As mentioned before, in this paper we consider the degenerate bivariate risk model defined in \eqref{eq_DefinitionModel}.
	In contrast to the assumptions in \cite{Azcue2018}, we allow that one initial capital is negative as long as the other is not, i.e., we consider the solvency set $\mathcal{S}:=\mathbb{R}^2\backslash\mathbb{R}^2_{<0}$. 
    Without loss of generality, we assume that the second branch is equally or less profitable than the first one, i.e.
    
	\begin{equation}\label{eq_Assumption_c1b1c2b2}
		\frac{c_1}{b_1} \geq \frac{c_2}{b_2}.
	\end{equation}
	Both branches pay dividends to their shareholders using part of their surpluses, where the dividend payment strategy $\mathbf{L}(t)$ represents the total amount of dividends paid by both branches up to time $t$. We define the associated controlled process with initial surplus $\mathbf{x}=(x_1,x_2)$ as 
	
	\begin{equation}\label{Definition_Model}
		\mathbf{X}^{\mathbf{L}}(t)  := \mathbf{X}(t)- \mathbf{L}(t)
		= \begin{pmatrix}
			X_1^{\mathbf{L}}(t) \\
			X_2^{\mathbf{L}}(t)
		\end{pmatrix}
		= \begin{pmatrix}
			X_1(t)-L_1(t) \\
			X_2(t)-L_2(t)
		\end{pmatrix}.
	\end{equation}
	
	This is mainly the model as the one considered in \cite{Azcue2018}. The crucial difference to our work is, that in \cite{Azcue2018} the authors consider the ruin time  
	
	\begin{equation*}
		\tau_\vee^{\mathbf{L}}:=\inf \left\{ t>0: X_1(t)-L_1(t)<0 \text{ or } X_2(t)-L_2(t)<0 \right\},
	\end{equation*}
	i.e., the first moment when \emph{at least one} branch faces ruin, while we consider the time of \emph{simultaneous ruin}
	
	\begin{equation*}
		\tau^{\mathbf{L}}:=\inf \left\{ t>0: X_1(t)-L_1(t)<0 \text{ and } X_2(t)-L_2(t)<0 \right\},
	\end{equation*}
	i.e., the first moment when \emph{both} branches face ruin. Apart from the \emph{at-least-one} ruin, the simultaneous ruin is a common assumption for bivariate risk models, which is considered e.g. in \cite{Avram2009}, \cite{Foss2017} or \cite{Hu2013}. \\
	Corresponding to the new solvency set, we extend the concept of admissibility of a dividend strategy in comparison with \cite{Azcue2018}: 
	
	\begin{definition}[Admissible dividend payment strategy]\label{def:dividend_strategy}
		A bivariate dividend payment strategy $\mathbf{L}(t)=(L_1(t), L_2(t))_{t \geq 0}$ is called \textbf{admissible} if it is a componentwise non-decreasing, càglàd stochastic process that is predictable with respect to the natural filtration $\mathbb{F}$ of $\mathbf{X}$ and if it satisfies 
		
		\begin{enumerate}[label=(\roman*)]
			\item $L_i(0)=0$ for $i=1,2$,
			\item if $X_i(t) < L_i(t)$ then $\diff L_i(t)=0$ for all $t\geq0,i=1,2$,
			\item if  $X_i(t) \geq L_i(t)$ then $X_i(t)+\diff X_i(t) \geq \diff L_i(t)$ for all $t \geq 0,i=1,2$.
		\end{enumerate} 
	\end{definition}
	The previous definition ensures that dividends are only paid as long as the associated controlled branch is non-negative. Moreover, it guarantees that no branch faces ruin due to dividend payments. Note that we set $L_i(t) \equiv L_i(\tau^{\mathbf{L}})$ for $t \ge \tau^{\mathbf{L}}$. \\
	Further, for any admissible dividend strategy $\mathbf{L}$ and at any time point $0\leq t < \tau^{\mathbf{L}}$ it holds that 
	
		\begin{equation}\label{bounds}
			L_1(t)\leq X_1(t) \qquad \text{or} \qquad L_2(t)\leq X_2(t),
		\end{equation} 
    which is a direct consequence of the definition of our ruin time.
    
	The set of all admissible dividend strategies with initial surplus $\mathbf{x}=(x_1,x_2) \in \mathcal{S}$ is denoted by $\Pi_{\mathbf{x}}$. Given any admissible dividend strategy $\mathbf{L} \in \Pi_{\mathbf{x}}$, the associated controlled process $\mathbf{X}^{\mathbf{L}}$ is adapted to the natural filtration $\mathbb{F}$ of $\mathbf{X}$ and the ruin time $\tau^{\mathbf{L}}$ is a stopping time with respect to this filtration. Moreover, the two-dimensional controlled risk process is of finite variation and by construction, ruin can only happen at the arrival of claims.
	
	For a fixed dividend strategy $\mathbf{L} \in \Pi_{\mathbf{x}}$, $\mathbf{x} \in \mathcal{S}$, the \emph{value function} $V^{\mathbf{L}}(\mathbf{x})$, which represents the \emph{cumulative expected discounted dividends}, is defined as 
	
	\begin{equation} \label{def:cumulative_expected_dividends}
		V^{\mathbf{L}}(\mathbf{x}) := \mathbb{E}_{\mathbf{x}} \left[ \int_{0}^{\tau^{\mathbf{L}}} e^{-qs} \diff \mathbf{L}(s)\right] = \mathbb{E}_{\mathbf{x}} \left[ \int_{0}^{\tau^{\mathbf{L}}} e^{-qs} \diff L_1(s) + \int_{0}^{\tau^{\mathbf{L}}} e^{-qs} \diff L_2(s) \right] ,
	\end{equation}
	where $q>0$ is a constant discount factor. On $\mathbb{R}^2_{<0}$ we assume $V^{\mathbf{L}}$ to be zero. Our goal is to identify the \emph{optimal value function} and the corresponding strategy, i.e., to solve the optimization problem defined by 
	
	\begin{equation}\label{optimalvaluefct}
		V(\mathbf{x})=\sup_{\mathbf{L} \in \Pi_{\mathbf{x}}} V^{\mathbf{L}}(\mathbf{x}) 
	\end{equation} 
	for any $\mathbf{x} \in \mathcal{S}$.

	\section{The optimal value function} \label{Section_OptimalValueFunction}
	We start our investigation of the optimal value function by collecting some properties that are going to be used later on. We emphasize that most results and proofs in this section are in large parts similar to either the one-dimensional case presented in \cite{Azcue2014} or the degenerate model considered in \cite{Azcue2018}. Therefore, our argumentations focus on the differences. 
	
	\begin{lemma}\label{lemma:V_properties_lower_upper_bound}
		For all $(x_1,x_2) \in \mathcal{S}$ the optimal value function is well-defined. If both $x_1 ,x_2\geq 0$, then the optimal value function satisfies 
		
		\begin{equation} \label{eq_SimpleBound_1}
			x_1+x_2+\frac{c_1+c_2}{q+\lambda} \le V(x_1,x_2) \le x_1 + x_2 + \frac{c_1+c_2}{q}.
		\end{equation}
		If $x_1 \ge 0, x_2 < 0$ the optimal value function satisfies
		
		\begin{equation} \label{eq_SimpleBound_2}
			x_1 + \frac{c_1}{q+\lambda} + \frac{c_2}{q+ \lambda}e^{(q+\lambda) \frac{x_2}{c_2}} \le V(x_1,x_2) \le x_1 + \frac{c_1}{q} + \frac{c_2}{q} e^{q \frac{x_2}{c_2}}.
		\end{equation}
		Equation \eqref{eq_SimpleBound_2} analogously holds for the case $x_1<0$ and $x_2 \geq 0$ with indices exchanged. 
	\end{lemma}
	
	\begin{proof}
		The proofs for well-definedness and inequality \eqref{eq_SimpleBound_1} are analog to the one-dimensional case presented in \cite[Prop. 1.2]{Azcue2014}. 
		To show \eqref{eq_SimpleBound_2} assume that $x_1 \geq 0$ and $x_2 <0$. Define the strategy $\mathbf{L}^0 = (L_1^0,L_2^0)$ as the strategy that pays at every $t\geq 0$ the maximum dividends possible. Obviously the ruin time $\tau^{\mathbf{L}_0}$ is equal to the arrival time of the first claim, denoted by $\tau_1 \sim \operatorname{Exp}(\lambda)$. In the first branch we obtain  $L_1(t) = x_1 + c_1 t$ for $t \leq \tau_1$. The second branch can only pay dividends if $X_2(t) \geq L_2(t)$. Hence, we need to wait at least until $t_0 = -\frac{x_2}{c_2}$, before any dividend can be paid. Consequently, the resulting dividend process is given by $L_2(t) = (x_2 + c_2 t) \mathds{1}_{\{ -\frac{x_2}{c_2} \leq t \leq \tau_1\}}$. We get 
		
		\begin{align*}
			V^{\mathbf{L}_0}(x_1, x_2) & = \mathbb{E}_{\mathbf{x}} \left[ \int_{0}^{\tau_1} e^{-qt} \diff L_1(t) + \int_{0}^{\tau_1} e^{-qt} \diff L_2(t)\right] \\
			& = x_1 + c_1 \mathbb{E}_{\mathbf{x}} \left[ \int_{0}^{\tau_1} e^{-qt} \diff t \right] + c_2 \mathbb{E}_{\mathbf{x}} \left[ \mathds{1}_{\{\tau_1 \ge -\frac{x_2}{c_2 }\}} \int_{-\frac{x_2}{c_2}}^{\tau_1} e^{-qt} \diff t \right] \\
			& = x_1 + \frac{c_1}{q+\lambda} + \frac{c_2}{q+\lambda} e^{(q+\lambda) \frac{x_2}{c_2}} ,
		\end{align*}
		which, by \eqref{optimalvaluefct}, implies the lower bound. The upper bound follows by a similar computation, since for any admissible dividend strategy $\mathbf{L}=(L_1,L_2)$ we have $L_1(t) \le (x_1 + c_1 t) \mathds{1}_{\{t \ge 0\}}$ and $L_2(t) \leq (x_2 + c_2 t) \mathds{1}_{\{t\ge -\frac{x_2}{c_2}\}}$.
	\end{proof}
	
	\begin{lemma}\label{Lemma_V_properties_increasing_locally_Lipschitz}
		The optimal value function V is componentwise increasing, locally Lipschitz, and satisfies 
		
		\begin{equation}\label{eq_increments1}
			0 < V(x_1 + h, x_2)-V(x_1,x_2) \le (e^{(q+\lambda)h/c_1}-1) V(x_1,x_2),
		\end{equation} 
		and 
		
		\begin{equation}\label{eq_increments2}
			0 < V(x_1, x_2+h)-V(x_1,x_2) \le (e^{(q+\lambda)h/c_2}-1) V(x_1,x_2)
		\end{equation}
		for any initial surplus $(x_1,x_2) \in \mathcal{S}$ and any $h>0$. \\
		In the case that $x_1 \ge 0$ and $x_2 \in \mathbb{R}$, we get
		
		\begin{equation}\label{eq_increments1lowerbound}
			h \le V(x_1 + h, x_2)-V(x_1,x_2),
		\end{equation}
		and if $x_1 \in \mathbb{R}$ and $x_2 \ge 0$, we get
		
		\begin{equation}\label{eq_increments2lowerbound}
			h \le V(x_1, x_2+h)-V(x_1,x_2),
		\end{equation}
		for any $h >0$. 
	\end{lemma}
	
	\begin{proof}
		The property of being componentwise (non-strictly) increasing, the upper bounds in equations \eqref{eq_increments1} and \eqref{eq_increments2}, as well as equations \eqref{eq_increments1lowerbound} and \eqref{eq_increments2lowerbound} can be proven exactly as in the proof of \cite[Lemma 3.2]{Azcue2018}. Hence, we only show the lower bounds in \eqref{eq_increments1} and \eqref{eq_increments2}, where due to analogy, we restrict to the former. If $x_1\geq 0$, $x_2\in\mathbb{R}$, then the statement is a direct consequence of \eqref{eq_increments1lowerbound}. Hence, we need to show that for $x_2\geq 0,~h>0, ~x_1<0$ it holds that 
		
		\begin{equation*}
			0 < V(x_1+h,x_2) - V(x_1,x_2). 
		\end{equation*}
	    W.l.o.g. we may assume that $x_1+h\le0$ because otherwise, it follows that 
	    
	    \begin{equation*}
	        V(x_1+h, x_2) - V(x_1,x_2) \geq V(x_1+h,x_2) -V(0,x_2) \geq x_1+h > 0,
	    \end{equation*} since $V$ is increasing and due to \eqref{eq_increments1lowerbound}.
		The rest of the proof is done by construction. Given an $\varepsilon>0$, let $\mathbf{L}^0=(L_1^0,L_2^0)\in\Pi_{x_1,x_2}$ be a strategy such that
		
		\begin{equation*}
			V^{\mathbf{L}^0}(x_1,x_2) > V(x_1,x_2) - \varepsilon.	
		\end{equation*} 
		We define a new strategy $\mathbf{L}=(L_1(t),L_2(t))\in\Pi_{x_1+h,x_2}$ as follows:
		
		\begin{itemize}
			\item Dividends from branch two are paid according to strategy $L_2^0$, i.e. $L_2(t) := L_2^0(t)$.
			\item Branch one pays no dividends as long as its uncontrolled surplus is below $+h$. Once the surplus hits $h$, it pays immediately $h$ as a lump sum, setting the controlled surplus to $0$. Afterward, it continues by paying dividends according to $L_1^0$. 
		\end{itemize}
		By construction we have $\big(\tau^{\mathbf{L}}\big|\mathbf{X}(0) = (x_1+h,x_2)\big) = \big(\tau^{\mathbf{L}^0}\big|\mathbf{X}(0) = (x_1,x_2)\big)$. Moreover, for any $s\geq 0$ we define the stopping time $\tau_s$ as the first time the process $X_1$ reaches $s$, i.e.
		
		\begin{equation*}
			\tau_s :=  \inf\left\{ 0\leq t\leq \tau^{\mathbf{L}}: ~X_1(t) = s	\right\}, 
		\end{equation*}
		where $\inf \emptyset = \infty$. Then it holds that
		
		\begin{align*} 
			\mathbb{P}_{x_1+h,x_2}(\tau_h \in {}\cdot{}) = \mathbb{P}_{x_1,x_2}(\tau_0 \in {}\cdot{}).
		\end{align*}  We conclude 
		
		\begin{align*}
			V^{\mathbf{L}}(x_1+h,x_2) 
			&= \mathbb{E}_{x_1,x_2}\left[\int_0^{\tau^{\mathbf{L}^0}}e^{-qs} \diff L_1^0(s) +\int_0^{\tau^{\mathbf{L}^0}} e^{-qs} \diff L_2^0(s) \right] + \mathbb{E}_{x_1+h,x_2}\left[h\cdot e^{-q\cdot \tau_h} \right] \\
			&= V^{\mathbf{L}^0}(x_1,x_2) + \mathbb{E}_{x_1,x_2}\left[h\cdot e^{-q\cdot \tau_0} \right]. 
		\end{align*}
		By construction follows that 
		
		\begin{align*}
			V(x_1+h,x_2) \geq V^{\mathbf{L}}(x_1+h,x_2) &= V^{\mathbf{L}^0}(x_1,x_2) + \mathbb{E}_{x_1,x_2}\left[h\cdot e^{-q\cdot \tau_0} \right]  \\
			&> V(x_1,x_2) -\varepsilon + \mathbb{E}_{x_1,x_2}\left[h\cdot e^{-q\cdot \tau_0} \right] ,	
		\end{align*}
		which implies 
		
		\begin{equation*}
			V(x_1+h,x_2) \geq V(x_1,x_2) +  \mathbb{E}_{x_1,x_2}\left[h\cdot e^{-q\cdot \tau_0} \right]
		\end{equation*}
		since $\varepsilon>0$ was arbitrary. We note that $\mathbb{E}_{x_1,x_2}\left[h\cdot e^{-q\cdot \tau_0} \right]>0$, since the probability that $\tau_0$ is reached before the arrival of the first claim is strictly positive. Thus the proof is finished. 
	\end{proof}
	
	Our next result is the \emph{Dynamic Programming Principle}, which heuristically states that an optimal strategy must be optimal at any point in time. 
	
	\begin{proposition}[Dynamic Programming Principle] \label{Proposition_DPP}
		For any initial surplus $\mathbf{x} \in \mathcal{S}$ and any stopping time $\tau$ it holds that
		
		\begin{align*}
			V(\mathbf{x}) & = \sup_{\mathbf{L} \in \Pi_{\mathbf{x}}} \mathbb{E}_{\mathbf{x}} \Bigg[ \int_{0}^{\tau \wedge \tau^{\mathbf{L}}} e^{-qs} \diff L_1(s) + \int_{0}^{\tau \wedge \tau^{\mathbf{L}}} e^{-qs} \diff L_2(s) \\ 
			&\qquad \ \qquad \ \qquad   + e^{-q(\tau \wedge \tau^{\mathbf{L}})} V\left(X_1^{\mathbf{L}}(\tau \wedge \tau^{\mathbf{L}}),X_2^{\mathbf{L}}(\tau \wedge \tau^{\mathbf{L}})\right)\Bigg].
		\end{align*}
	\end{proposition}
	
	\begin{proof}
		We use the same strategy as in the proof of \cite[Lemma 1.2]{Azcue2014}: We show the statement for a fixed time $T\geq 0$, and then the general case follows using the arguments given in \cite[Chapter II.2]{Zhu1992}. \\
		Set 
		
		\begin{align}
			v(\mathbf{x},T) & = \sup_{\mathbf{L} \in \Pi_{\mathbf{x}}} \mathbb{E}_{\mathbf{x}} \Biggl[ \int_{0}^{T \wedge \tau^{\mathbf{L}}} e^{-qs} \diff L_1(s) + \int_{0}^{T \wedge \tau^{\mathbf{L}}} e^{-qs} \diff L_2(s)  \nonumber \\
			& \mathrel{\phantom{= \sup_{\mathbf{L} \in \Pi_{\mathbf{x}}} \mathbb{E}_{\mathbf{x}} \Biggl[}}  + e^{-q(T \wedge \tau^{\mathbf{L}})} V\left(X_1^{\mathbf{L}}(T \wedge \tau^{\mathbf{L}}),X_2^{\mathbf{L}}(T \wedge \tau^{\mathbf{L}})\right)\Biggr]. \label{Proof_Proposition_DPP_0}
		\end{align}
		Then, by the same arguments as in \cite[Lemma 1.2]{Azcue2014}, we have for any $\mathbf{x}\in\mathcal{S}$ and $\mathbf{L}\in\Pi_{\mathbf{x}}$ that 
		
		\begin{equation*}
			V^{\mathbf{L}}(\mathbf{x}) \leq v(\mathbf{x},T).
		\end{equation*}
		The proof of the inverse inequality $v(\mathbf{x},T)\geq  V^{\mathbf{L}}(\mathbf{x})$ is also similar to \cite{Azcue2014}, but more involved due to the new cases that appear if one initial capital is strictly negative. Therefore we will go into detail here: \\   
		For any given $\varepsilon > 0$, we fix $\mathbf{L} \in \Pi_{\mathbf{x}}$ such that 
		
		\begin{align}\label{Proof_Proposition_DPP_1}
			\mathbb{E}_{\mathbf{x}} \Biggl[ \int_{0}^{T \wedge \tau^{\mathbf{L}}} e^{-qs} \diff L_1(s) + \int_{0}^{T \wedge \tau^{\mathbf{L}}} e^{-qs} \diff L_2(s) \nonumber \\ 
			\mathrel{\phantom{\mathbb{E}_{\mathbf{x}} \Biggl[}} + e^{-q(T \wedge \tau^{\mathbf{L}})} V\left(X_1^{\mathbf{L}}(T \wedge \tau^{\mathbf{L}}), X_2^{\mathbf{L}}(T \wedge \tau^{\mathbf{L}})\right) \Biggr] 
			& \ge v(\mathbf{x}, T) - \frac{\varepsilon}{2}.
		\end{align}
		By Lemma \ref{Lemma_V_properties_increasing_locally_Lipschitz} the optimal value function $V$ is increasing and continuous in $\mathcal{S}$, and hence we can construct monotonically increasing  sequences $(v_1^{(i)})_{i \in \mathbb{N}}$, $(v_2^{(i)})_{i \in \mathbb{N}}$ with $(v_1^{(1)},v_2^{(1)}) \in \mathcal{S} $, $0 \in (v_1^{(i)})_{i \in \mathbb{N}}$, $0 \in (v_2^{(i)})_{i \in \mathbb{N}}$ and $\lim_{i \rightarrow \infty} v_1^{(i)} = \lim_{i \rightarrow \infty} v_2^{(i)} = \infty$ such that if $y_1 \in [v_1^{(i)}, v_1^{(i+1)})$ or $y_2 \in [v_2^{(j)}, v_2^{(j+1)})$, then 
		
		\begin{equation*}
			V(y_1,x_2) - V(v_1^{(i)},x_2) < \frac{\varepsilon}{8} \quad \forall x_2 \in \mathbb{R}\quad \text{ and } \quad V(x_1, y_2) - V(x_1, v_2^{(j)}) < \frac{\varepsilon}{8} \quad \forall x_1 \in \mathbb{R}
		\end{equation*} 
		for $i\geq 0$, $j\geq 0$. W.l.o.g., in the following we solely use subscript $i$ to simplify the notation. This is possible as we can always insert additional elements into the sequences. Hence, we get 
		
		\begin{equation} \label{Proof_Proposition_DPP_2}
		    V(y_1,y_2) < V(v_1^{(i)},y_2) + \frac{\varepsilon}{8} < V(v_1^{(i)}, v_2^{(i)}) + \frac{\varepsilon}{4}.
		\end{equation} 
		
		Given $i\in\mathbb{N}$ we consider strategies $\mathbf{L}^i = (L_1^i(t), L_2^i(t))_{t \ge 0} \in \Pi_{(v_1^{(i)}, v_2^{(i)})}$ such that 
		
		\begin{equation}\label{Proof_Proposition_DPP_2b}
		    V(v_1^{(i)}, v_2^{(i)}) - V^{\mathbf{L}^i}(v_1^{(i)}, v_2^{(i)}) \leq \frac{\varepsilon}{4}.
		\end{equation}
		Based on these strategies we define a new strategy $\mathbf{L}^* = (L_1^*(t), L_2^*(t))_{t \ge 0}$ as follows: 
		
		\begin{itemize}
			\item If $\tau^{\mathbf{L}} \leq T$, set $L_1^*(t) = L_1(t)$ and $L_2^*(t) = L_2(t)$ for all $t \ge 0$. 
			\item If $\tau^{\mathbf{L}} > T$, set $L_1^*(t) = L_1(t)$ and $L_2^*(t) = L_2(t)$ for all $t \in [0,T]$. 
			\item If $\tau^{\mathbf{L}} > T$ and $t\geq T$, choose $i$ such that $X_1^{\mathbf{L}}(T) \in [v_1^{(i)}, v_1^{(i+1)})$ and $X_2^{\mathbf{L}}(T) \in [v_2^{(i)}, v_2^{(i+1)})$. We distinguish three cases: 
			\begin{itemize}
				\item If $X_1^{\mathbf{L}}(T) \geq 0$ and $X_2^{\mathbf{L}}(T) \geq 0$, then by assumption we have $v_1^{(i)} \geq 0$ and $v_2^{(i)} \geq 0$. In $\mathbf{L}^*$, branch one pays immediately $X_1^{\mathbf{L}}(T)- v_1^{(i)}$ and branch two pays immediately $X_2^{\mathbf{L}}(T)  - v_2^{(i)}$ as dividends at time $T$. Afterward we follow $\mathbf{L}^i$. 
				\item If $X_1^{\mathbf{L}}(T) \ge 0$ and $X_2^{\mathbf{L}}(T) < 0$ then by assumption $v_1^{(i)} \ge 0$. In $\mathbf{L}^*$, branch one pays immediately $X_1^{\mathbf{L}}(T)-v_1^{(i)}$ as dividends. Then we follow $\mathbf{L}^i$ from surplus $(v_1^{(i)}, X_2^{\mathbf{L}}(T))$.
				\item Similar to the previous case, if $X_1^{\mathbf{L}}(T) < 0$ and $X_2^{\mathbf{L}}(T) \ge 0$, branch two pays $X_2^{\mathbf{L}}(T)-v_2^{(i)}$ as dividends and then we follow $\mathbf{L}^i$ from surplus $(X_1^{\mathbf{L}}(T), v_2^{(i)})$.
			\end{itemize}
		\end{itemize}
		In the case of $\tau^{\mathbf{L}} > T$ and $t\geq T$, assuming that $X_1^{\mathbf{L}}(T) \in [v_1^{(i)}, v_1^{(i+1)})$ and $X_2^{\mathbf{L}}(T) \in [v_2^{(i)}, v_2^{(i+1)})$, it follows by \eqref{Proof_Proposition_DPP_2b}
	
			\begin{align} 
				V^{\mathbf{L}^*}(X_1^{\mathbf{L}}(T), X_2^{\mathbf{L}}(T)) &= \begin{cases}
					X_1^{\mathbf{L}}(T)- v_1^{(i)} + X_2^{\mathbf{L}}(T)- v_2^{(i)} + V^{\mathbf{L}_i}(v_1^{(i)}, v_2^{(i)}), & v_1^{(i)}\geq 0 , ~v_2^{(i)} \geq 0, \\ 
					X_1^{\mathbf{L}}(T)- v_1^{(i)} +  V^{\mathbf{L}_i}(v_1^{(i)}, X_2^{\mathbf{L}}(T)), &  v_1^{(i)}\geq 0,~v_2^{(i)}<0, \\
					X_2^{\mathbf{L}}(T)- v_2^{(i)} +  V^{\mathbf{L}_i}(X_1^{\mathbf{L}}(T), v_2^{(i)}), & v_1^{(i)}<0,~v_2^{(i)} \geq 0
				\end{cases} \nonumber \\
				& \ge \mathds{1}_{\{v_1^{(i)} \ge 0\}}\left(X_1^{\mathbf{L}}(T)- v_1^{(i)}\right) + \mathds{1}_{\{v_2^{(i)} \ge 0\}}\left(X_2^{\mathbf{L}}(T)- v_2^{(i)}\right) + V^{\mathbf{L}_i}(v_1^{(i)}, v_2^{(i)}) \nonumber  \\
				& \ge V^{\mathbf{L}_i}(v_1^{(i)}, v_2^{(i)}) \nonumber \\
				& \ge V(v_1^{(i)}, v_2^{(i)}) - \frac{\varepsilon}{4}. \label{Proof_Proposition_DPP_3}
			\end{align}
		Strategy $\mathbf{L}^*$ is then admissible and its value function can be obtained as 	
		
		\begin{align*}
			V^{\mathbf{L}^*}(\mathbf{x}) & = \mathbb{E}_{\mathbf{x}} \left[ \int_{0}^{T \wedge \tau^{\mathbf{L}}} e^{-qs} \diff L_1(s) + \int_{0}^{T \wedge \tau^{\mathbf{L}}} e^{-qs} \diff L_2(s) \right. \\
			& \qquad \ \qquad  \left. + \int_{T \wedge \tau^{\mathbf{L}}}^{\tau^{\mathbf{L}^*}} e^{-qs} \diff L_1^*(s) + \int_{T \wedge \tau^{\mathbf{L}}}^{\tau^{\mathbf{L}^*}} e^{-qs} \diff L_2^*(s) \right] \\
			& = \mathbb{E}_{\mathbf{x}} \left[ \int_{0}^{T \wedge \tau^{\mathbf{L}}} e^{-qs} \diff L_1(s) + \int_{0}^{T \wedge \tau^{\mathbf{L}}} e^{-qs} \diff L_2(s) \right. \\
			&\qquad \ \qquad \left. + e^{-q(T \wedge \tau^{\mathbf{L}})} \mathbb{E}_{\mathbf{X}^{\mathbf{L}}(T \wedge \tau^{\mathbf{L}})} \left[\int_{0}^{\tau^{\mathbf{L}^*}-(T \wedge \tau^{\mathbf{L}})} e^{-qs} \diff L_1^*(s+(T \wedge \tau^{\mathbf{L}})) \right. \right. \\
			& \qquad \ \qquad \ \qquad \ \qquad \ \qquad \ \qquad \ \quad \left. \left. + \int_{0}^{\tau^{\mathbf{L}^*}-(T \wedge \tau^{\mathbf{L}})} e^{-qs} \diff L_2^*(s+(T \wedge \tau^{\mathbf{L}})) \right] \right] \\
			& = \mathbb{E}_{\mathbf{x}} \Biggl[ \int_{0}^{T \wedge \tau^{\mathbf{L}}} e^{-qs} \diff L_1(s) + \int_{0}^{T \wedge \tau^{\mathbf{L}}} e^{-qs} \diff L_2(s) \\
			& \qquad \ \qquad + e^{-q(T \wedge \tau^{\mathbf{L}})} V^{\mathbf{L}^*}\left(X_1^{\mathbf{L}}(T \wedge \tau^{\mathbf{L}}), X_2^{\mathbf{L}}(T \wedge \tau^{\mathbf{L}}) \right) \Biggr].
		\end{align*}
		Now, under usage of  \eqref{Proof_Proposition_DPP_1}, \eqref{Proof_Proposition_DPP_2}, \eqref{Proof_Proposition_DPP_3} and Lemma \ref{lemma:V_properties_lower_upper_bound} the result follows, since 
		
		\begin{align*}
			v&(\mathbf{x},T) - V^{\mathbf{L}^*}(\mathbf{x}) \\ 
			& \le \mathbb{E}_{\mathbf{x}} \Biggl[ \int_{0}^{T \wedge \tau^{\mathbf{L}}} e^{-qs} \diff L_1(s) + \int_{0}^{T \wedge \tau^{\mathbf{L}}} e^{-qs} \diff L_2(s) + e^{-q(T \wedge \tau^{\mathbf{L}})} V\left(X_1^{\mathbf{L}}(T \wedge \tau^{\mathbf{L}}), X_2^{\mathbf{L}}(T \wedge \tau^{\mathbf{L}})\right) \Biggr] \\ 
			& \qquad -\mathbb{E}_{\mathbf{x}}\Biggl[ \int_{0}^{T \wedge \tau^{\mathbf{L}}} e^{-qs} \diff L_1(s) + \int_{0}^{T \wedge \tau^{\mathbf{L}}} e^{-qs} \diff L_2(s)   \\
			& \qquad \ \qquad + e^{-q(T \wedge \tau^{\mathbf{L}})} V^{\mathbf{L}^*}\left(X_1^{\mathbf{L}}(T \wedge \tau^{\mathbf{L}}), X_2^{\mathbf{L}}(T \wedge \tau^{\mathbf{L}}) \right) \Biggr] + \frac{\varepsilon}{2} \\
			& = \mathbb{E}_{\mathbf{x}} \left[ e^{-q(T \wedge \tau^{\mathbf{L}})} \left( V\left(X_1^{\mathbf{L}}(T \wedge \tau^{\mathbf{L}}), X_2^{\mathbf{L}}(T \wedge \tau^{\mathbf{L}})\right) - V^{\mathbf{L}^*}\left(X_1^{\mathbf{L}}(T \wedge \tau^{\mathbf{L}}), X_2^{\mathbf{L}}(T \wedge \tau^{\mathbf{L}})\right) \right) \right] + \frac{\varepsilon}{2}\\
			& \le  \mathbb{E}_{\mathbf{x}} \left[ e^{-q(T \wedge \tau^{\mathbf{L}})} \left( V\left(X_1^{\mathbf{L}}(T \wedge \tau^{\mathbf{L}}), X_2^{\mathbf{L}}(T \wedge \tau^{\mathbf{L}})\right) - \left( V(v_1^{(i)}, v_2^{(i)}) -\frac{\varepsilon}{4} \right) \right) \right] + \frac{\varepsilon}{2} \\
			& < \mathbb{E}_{\mathbf{x}} \left[ e^{-q(T \wedge \tau^{\mathbf{L}})} \left( \frac{\varepsilon}{4} +\frac{\varepsilon}{4} \right) \right] + \frac{\varepsilon}{2} \le \varepsilon
		\end{align*} and because $\varepsilon$ was arbitrary. 
	\end{proof}

    We now aim to derive the HJB equation. Therefore, recall the concept of the \textit{discounted infinitesimal generator}, cf. \cite[Sec. 1.4]{Azcue2014}, \cite[Eq. (7)]{Azcue2018}: Given a Markov process $\mathbf{S}$ in $\mathbb{R}^2$ and $\mathbf{x}\in\mathcal{S}$, set 
    
	\begin{equation*}
		\tilde{\mathcal{G}}(\mathbf{S},f)(\mathbf{x}) = \lim_{t\downarrow 0} \frac{\mathbb{E}_{\mathbf{x}}[e^{-qt}f(\mathbf{S}_t)]-f(\mathbf{x})}{t}
	\end{equation*}
	for any real-valued, continuously differentiable function $f$ on $\mathcal{S}$ such that the above limit exists. For our considerations we choose $\mathbf{S}(t) = \mathbf{X}^{\mathbf{L}}(t\wedge \tau^{\mathbf{L}})$, i.e., the controlled risk process stopped at ruin. \\
	Let $\ell_1,\ell_2\geq 0$ be constants and define the dividend strategy $\mathbf{L}$, that constantly pays dividends at rate $\ell_1,~\ell_2$ from branch one and two, whenever the respective surplus is non-negative. Then clearly $\mathbf{L}$ is admissible. Further, let $\tau_1$ be the first claim arrival time of $\mathbf{X}$. Using the same arguments as in \cite[Section 1.4]{Azcue2014}, we derive 
	
	\begin{equation} 
		\label{eq_discounted_infinitesimal_generator} 
		\begin{aligned}
			\tilde{\mathcal{G}} & \left(\left(\mathbf{X}^{\mathbf{L}}(t \wedge \tau^{\mathbf{L}})\right)_{t \ge 0}, f \right)(\mathbf{x})  = \tilde{\mathcal{G}} \left(\left(\mathbf{X}^{\mathbf{L}}(t \wedge \tau_1)\right)_{t \ge 0}, f \right)(\mathbf{x}) \\
			& = \left(c_1-\ell_1 \mathds{1}_{\{x_1 \ge 0\}}\right) f_{x_1}(\mathbf{x}) + \left(c_2-\ell_2 \mathds{1}_{\{x_2 \ge 0\}}\right) f_{x_2}(\mathbf{x}) - (\lambda+q) f(\mathbf{x}) + \lambda \mathcal{I}(f) (\mathbf{x}),
		\end{aligned}
	\end{equation}
	where $\mathcal{I}$ is an integral operator given via
	
	\begin{equation}
		\label{eq_IntegralOperator}
		\mathcal{I}(f) (\mathbf{x}) := \int_{0}^{(x_1/b_1) \vee (x_2/b_2)} f(x_1 -b_1 \alpha, x_2 -b_2 \alpha) \diff F(\alpha).
	\end{equation}
	Moreover, set 
	
	 \begin{equation}
		\label{eq_Definition_mathcalL}
		\mathcal{L}(V)(\mathbf{x}) := c_1 V_{x_1}(\mathbf{x}) + c_2 V_{x_2}(\mathbf{x}) - (q+\lambda) V(\mathbf{x}) + \lambda \mathcal{I}(V)(\mathbf{x}).
	\end{equation}
	The HJB equation, i.e. the integro-differential equation which is satisfied by the optimal value function, turns out to be 
	
		\begin{equation}
	    \label{eq_HJB}
		\max\{\mathds{1}_{\{x_1 \ge 0\}}(1-V_{x_1}(\mathbf{x})), \mathds{1}_{\{x_2 \ge 0\}}(1-V_{x_2}(\mathbf{x})), \mathcal{L}(V)(\mathbf{x})\} = 0
	\end{equation}
	for any $\mathbf{x}\in \mathcal{S}$ and in the following we explain its derivation: \\
	As the case $\mathbf{x}\in\mathbb{R}^2_{\geq 0}$ is completely similar to the derivation of the HJB equation in \cite{Azcue2018}, we will only explain the differences in the new case, where one surplus is strictly negative. Due to symmetry we consider w.l.o.g. $x_1>0,~x_2<0$. Assume that $V$ is continuously differentiable. Since $x_2<0$, Equations \eqref{eq_discounted_infinitesimal_generator} and \eqref{eq_IntegralOperator} reduce to
	
	\begin{align*}
		\tilde{\mathcal{G}} \left(\left(\mathbf{X}^{\mathbf{L}}(t \wedge \tau_1)\right)_{t \ge 0}, V \right) (\mathbf{x})
		& = (c_1-l_1) V_{x_1}(\mathbf{x}) + c_2 V_{x_2}(\mathbf{x}) - (\lambda+q) V(\mathbf{x}) + \lambda \mathcal{I}(V) (\mathbf{x}), \\
		\mathcal{I}(V) (\mathbf{x}) &= \int_{0}^{(x_1/b_1)} V(x_1 -b_1 \alpha, x_2 -b_2 \alpha) \diff F(\alpha). 
	\end{align*}
	Let $t>0$ such that $t < - \tfrac{x_2}{c_2}$ and such that $t<\tfrac{x_1}{l_1 - c_1}$, if $l_1>c_1$. From Proposition \ref{Proposition_DPP} with $\tau = t \wedge \tau_1 \leq \tau ^{\mathbf{L}}$ it follows that 
	
	\begin{align*}
		V(\mathbf{x}) & \geq \mathbb{E}_{\mathbf{x}} \Biggl[ \int_{0}^{t \wedge \tau_1} e^{-qs} \cdot  \ell_1 \diff s + \int_{0}^{t \wedge \tau_1} e^{-qs} \cdot \ell_2 \cdot \mathds{1}_{\{s \ge - \frac{x_2}{p_2}\}} \diff s \\
		& \qquad \ \qquad + e^{-q (t \wedge \tau_1)} V\left(X_1^{\mathbf{L}}(t \wedge \tau_1),X_2^{\mathbf{L}}(t \wedge \tau_1)\right) \Biggr].  \\
		&= \mathbb{E}_{\mathbf{x}} \left[ \int_{0}^{t \wedge \tau_1} e^{-qs} \cdot  \ell_1 \diff s + e^{-q (t \wedge \tau_1)} V\left(X_1^{\mathbf{L}}(t \wedge \tau_1),X_2^{\mathbf{L}}(t \wedge \tau_1)\right) \right]. 
	\end{align*} 
	This implies
	
	\begin{align*}
		0 &\geq \lim_{t \downarrow 0} \frac{\ell_1 \cdot \mathbb{E}_{\mathbf{x}} \left[ \int_{0}^{t \wedge \tau_1} e^{-qs} \diff s \right] + \mathbb{E}_{\mathbf{x}} \left[e^{-q (\tau_1 \wedge t)} V\left(X_1^{\mathbf{L}}(t \wedge \tau_1),X_2^{\mathbf{L}}(t \wedge \tau_1)\right) \right] - V(\mathbf{x})}{t}  \\
		& = \ell_1 + \tilde{\mathcal{G}}\left(\left(\mathbf{X}^{\mathbf{L}}(t \wedge \tau_1)\right)_{t \ge 0}, V\right)(\mathbf{x}),
	\end{align*}   
    and	we conclude
    
	\begin{equation*}
		\mathcal{L}(V)(\mathbf{x})	+ \ell_1 \cdot (1- V_{x_1}(\mathbf{x}))  \leq 0.
	\end{equation*}
	Lastly, choosing either $\ell_1=0$ or letting $\ell_1 \to\infty$ we obtain that 
	
	\begin{equation*}
		\max\left\{ 1-V_{x_1}(\mathbf{x}), \mathcal{L}(V)(\mathbf{x})\right\} \leq 0, 	
	\end{equation*}
	which is \eqref{eq_HJB} for $x_1>0, ~x_2<0$. The case $x_2 >0,~x_1 <0$ follows in complete analogy, where we obtain
	
	\begin{equation*}
		\max\left\{ 1-V_{x_2}(\mathbf{x}), \mathcal{L}(V)(\mathbf{x})\right\} \leq 0. 	
	\end{equation*}
	
	\begin{remark}
		Note that at first sight, Equations \eqref{eq_discounted_infinitesimal_generator}, \eqref{eq_IntegralOperator}, \eqref{eq_Definition_mathcalL}, and \eqref{eq_HJB} look very similar to equations (8), (9), (10) and (12) in \cite{Azcue2018}. The occurring indicator functions in \eqref{eq_discounted_infinitesimal_generator} and \eqref{eq_HJB} account only for the cases where one $x_i, ~i=1,2,$ is strictly negative, which implies that on $\mathbb{R}^2_{\geq 0}$ Equations \eqref{eq_discounted_infinitesimal_generator}, \eqref{eq_Definition_mathcalL} and \eqref{eq_HJB} superficially coincide with equations (8), (10) and (12) in \cite{Azcue2018}. The difference lies in the integral operator, as we need to allow to integrate up to the maximum instead of the minimum of $x_1/b_1$ and $x_2/b_2$ as $f$ is only assumed to be zero on the pure negative quadrant, while it may be strictly positive outside.  
	\end{remark}

	As usual for this type of problem, there are cases where the value function may not be differentiable and thus it may not fulfill the HJB equation in the classical sense. Thus, we use the notion of viscosity solutions in the following: We follow the definition given in \cite[Def. 3.4]{Azcue2018} 	and call a function $\underline{u} \colon \mathcal{S}^\circ \to \mathbb{R}$  \textbf{viscosity subsolution} of \eqref{eq_HJB} at a fixed point $\tilde{\mathbf{x}} \in \mathcal{S}^\circ$, if it is locally Lipschitz, and any continuously differentiable function $\psi \colon \mathcal{S}^\circ \to \mathbb{R}$ with $\psi(\tilde{\mathbf{x}}) = \underline{u}(\tilde{\mathbf{x}})$ such that $\underline{u} - \psi$ reaches its maximum in $\tilde{\mathbf{x}}$, satisfies 
	
		\begin{equation}
			\max\{\mathds{1}_{\{\tilde{x}_1 \ge 0\}}(1-\psi_{x_1}(\tilde{\mathbf{x}})), \mathds{1}_{\{\tilde{x}_2 \ge 0\}}(1-\psi_{x_2}(\tilde{\mathbf{x}})), \mathcal{L}(\psi)(\tilde{\mathbf{x}})\} \ge 0. \label{viscosity_subsolution}
		\end{equation}
	Moreover, a function $\overline{u} \colon \mathcal{S}^\circ \to \mathbb{R}$ is called a \textbf{viscosity supersolution} of \eqref{eq_HJB} at a fixed point $\tilde{\mathbf{x}} \in \mathcal{S}^\circ$, if it is locally Lipschitz, and any continuously differentiable function $\varphi \colon \mathcal{S}^\circ \to \mathbb{R}$ with $\varphi(\tilde{\mathbf{x}}) = \overline{u}(\tilde{\mathbf{x}})$ such that $\overline{u} - \varphi$ reaches its minimum in $\tilde{\mathbf{x}}$, satisfies 
	
		\begin{equation}
			\max\{\mathds{1}_{\{\tilde{x}_1 \ge 0\}}(1-\varphi_{x_1}(\tilde{\mathbf{x}})), \mathds{1}_{\{\tilde{x}_2 \ge 0\}}(1-\varphi_{x_2}(\tilde{\mathbf{x}})), \mathcal{L}(\varphi)(\tilde{\mathbf{x}})\} \le 0. \label{viscosity_supersolution}
		\end{equation}
	The functions $\psi$ and $\varphi$ are also called test functions. A function $u \colon \mathcal{S}^\circ \to \mathbb{R}$ is called \textbf{viscosity solution} at $\tilde{\mathbf{x}} \in \mathcal{S}^\circ$ if it is both a viscosity sub- and supersolution. 
	
	\begin{proposition} \label{Proposition_ViscositySolution}
		The optimal value function $V$ defined in \eqref{optimalvaluefct} is a viscosity solution of \eqref{eq_HJB} at any $\mathbf{x} \in \mathcal{S}^\circ$.
	\end{proposition}
	The proof of Proposition \ref{Proposition_ViscositySolution} is naturally split into two parts that cover sub- and supersolution, respectively. The proof that $V$ is a supersolution follows the idea of  \cite[Prop. 3.1]{Azcue2014}. Mainly, it uses the same arguments as the derivation of the HJB equation. For the sake of brevity, the details will be omitted here. Proving that $V$ is a viscosity subsolution follows the main ideas of \cite[Prop. 3.1]{Azcue2014} as well and is done by contradiction. As our enlarged solvency set demands some additional care, we go a little more into detail here. For an easier understanding, we split the proof into Lemma \ref{Lemma_Proof_SubSolution_Construction} and Lemma \ref{Lemma_Proof_SubSolution_Contradiction}, which yield an immediate contradiction, showing that $V$ is indeed a viscosity subsolution. \\
	In the following, for notational simplicity, we abuse notation and define
	
	\begin{equation*}
		[a_1,b_1]\times[a_2,b_2] := \emptyset, \text{ if } a_1>a_2 \text{ or } b_1>b_2.	
	\end{equation*}
	
	\begin{lemma}\label{Lemma_Proof_SubSolution_Construction}
		Assume $V$ is not a viscosity subsolution of \eqref{eq_HJB} at $\tilde{\mathbf{x}} = (\tilde{x}_1,\tilde{x}_2) \in \mathcal{S}^\circ$. Then we can find $\varepsilon >0$, 
		
		\begin{equation} \label{eq_hSets}
			h \in \begin{cases}
				\left(0,\frac{1}{2}(\abs{\tilde{x}_1} \wedge \abs{\tilde{x}_2})\right), & \text{ if } \abs{\tilde{x}_1}\wedge\abs{\tilde{x}_2} > 0, \\
				\left(0, \frac{1}{2}(\abs{\tilde{x}_1} \vee \abs{\tilde{x}_2})\right), & \text{ if } \abs{\tilde{x}_1} \wedge \abs{\tilde{x}_2}=0,
			\end{cases} 
		\end{equation}
		and a continuously differentiable function $\psi \colon \mathbb{R}^2 \to \mathbb{R}$ such that $\psi$ is a test function for a subsolution of equation \eqref{eq_HJB} satisfying 
		
		\begin{align}
			\label{eq_Proof_SubSolution_Construction1}
			\mathds{1}_{\{\tilde{x}_1 \ge 0\}} (1 - \psi_{x_1}(\mathbf{x})) &\le 0 && \text{for } \mathbf{x} \in [0 \wedge \tilde{x}_1-h, \tilde{x}_1 + h] \times (-\infty,\tilde{x}_2+h],\\
			\label{eq_Proof_SubSolution_Construction2} 
			\mathds{1}_{\{\tilde{x}_2 \ge 0\}} (1 - \psi_{x_2}(\mathbf{x})) &\le 0 && \text{for } \mathbf{x} \in (-\infty,\tilde{x}_1+h] \times [0 \wedge \tilde{x}_2-h, \tilde{x}_2 + h], \\
			\label{eq_Proof_SubSolution_Construction3}
			\mathcal{L} (\psi) (\mathbf{x}) &\le -2 \varepsilon q && \text{for } \mathbf{x} \in [\tilde{x}_1 - h, \tilde{x}_1 + h] \times [\tilde{x}_2 - h, \tilde{x}_2 + h], \\
			\label{eq_Proof_SubSolution_Construction4}		
			V(\mathbf{x}) &\leq \psi(\mathbf{x}) - 2 \varepsilon && \text{for } \mathbf{x} \in B^*:=  B_1 \cup B_2 \cup B_3 \cup B_4 \cup \mathbb{R}_{<0}^2, 
		\end{align}
		where
		
		\begin{align*}
			B_1 &:=  \left(-\infty,\tilde{x}_1 +h \right) \times \left(-\infty, \tilde{x}_2 - \frac{h}{2}\right), & B_2 &:=  \left(-\infty, \tilde{x}_1 - \frac{h}{2}\right) \times \left(-\infty,\tilde{x}_2 + h \right), \\
			B_3&:= \{\tilde{x}_1 + h\} \times (-\infty, \tilde{x}_2 +h], & B_4 &:=  (-\infty, \tilde{x}_1 +h] \times \{\tilde{x}_2 + h\}.
		\end{align*}		
	\end{lemma}
	Note that the minima ($\wedge$) inside the intervals in \eqref{eq_Proof_SubSolution_Construction1} and \eqref{eq_Proof_SubSolution_Construction2} are only relevant if $\tilde{x}_1=0$ or $\tilde{x}_2=0$, respectively: If $\tilde{x}_i>0$ then \eqref{eq_hSets} ensures that $\tilde{x}_i-h>0$. 
	
	\begin{proof}[Proof of Lemma \ref{Lemma_Proof_SubSolution_Construction}]
        The proof follows the outline of the univariate case as presented in \cite[Prop. 3.1]{Azcue2014} and consists in constructing a test function $\psi$ that fulfills the desired properties.
	\end{proof}

	\begin{lemma}\label{Lemma_Proof_SubSolution_Contradiction}
		Assume $V$ is not a viscosity subsolution and let $\psi, ~\tilde{\mathbf{x}}, ~h, ~\epsilon$ be as in Lemma \ref{Lemma_Proof_SubSolution_Construction}. Then it holds that 
		\begin{equation*}
			V(\tilde{\mathbf{x}}) < \psi(\tilde{\mathbf{x}})	.
		\end{equation*}
	\end{lemma}
	
	\begin{proof}[Proof of Lemma \ref{Lemma_Proof_SubSolution_Contradiction}]
		Also this proof follows the outline given in \cite[Prop. 3.1]{Azcue2014} and is similar to the proof of \cite[Thm. 3.5]{Azcue2018}. However, there are some decisive extra arguments needed in our extended setting, which is why we go a bit into detail here. \\
		Since $\psi$ is continuously differentiable, for any compact set $M\subset\mathbb{R}^2$ we can find $C\geq 0$ such that for any $\mathbf{x}\in M$ we have
		
		\begin{equation} \label{eq_Proof_M_compact_bounded}
			\mathcal{L}(\psi)(\mathbf{x}) \le C.
		\end{equation}
		Choose 
		
		\begin{equation} \label{eq_Proof_ThetaInequality}
			0<\theta < \min\left\{\frac{\varepsilon}{2 C}, \frac{1}{4q}, \frac{h}{2 \max\{c_1,c_2\}}\right\}
		\end{equation}
		and fix an admissible dividend policy $\mathbf{L}(t) = (L_1(t),L_2(t))\in\Pi_{\tilde{\mathbf{x}}}$ for all $t$.
		Similar to the proof of \cite[Thm. 3.5]{Azcue2018} we define the stopping times
		
		\begin{align*}
			\overline{\tau} & := \inf \{t > 0 \colon X_1^{\mathbf{L}}(t) \ge \tilde{x}_1 + h \quad \text{or} \quad X_2^{\mathbf{L}}(t) \ge \tilde{x}_2 + h\},\\
			\underline{\tau} & := \inf \{t > 0 \colon X_1^{\mathbf{L}}(t) \le \tilde{x}_1 - h \quad \text{or} \quad X_2^{\mathbf{L}}(t) \le \tilde{x}_2 - h\}
		\end{align*}and set 
		
		\begin{equation*}
			\tau^* := \overline{\tau} \wedge (\underline{\tau} + \theta) \wedge \tau^{\mathbf{L}},
		\end{equation*}
		where $\tau^{\mathbf{L}}$ is, as usual, the ruin time of the controlled process $\mathbf{X}^{\mathbf{L}}$. Clearly, $\tau^*< \infty$ for $h$ small enough. 
		Recall the sets $B^*,B_1,B_2,B_3$, and $B_4$ from Lemma \ref{Lemma_Proof_SubSolution_Construction}. Then by construction we have $\mathbf{X}^{\mathbf{L}}(\uline{\tau}+ \theta) \in B_1\cup B_2$, $\mathbf{X}^{\mathbf{L}}(\overline{\tau}) \in B_3\cup B_4$ and $\mathbf{X}^{\mathbf{L}}(\tau^{\mathbf{L}}) \in \mathbb{R}_{<0}^2.$ Hence, 
		
		\begin{equation} \label{eq_Proof_SubSolution_Contradiction_Proof1}
			\mathbf{X}^{\mathbf{L}}(\tau^*) \in B^*
		\end{equation}
		and consequently \eqref{eq_Proof_SubSolution_Construction4} implies
		
		\begin{equation} \label{eq_Proof_SubSolution_Contradiction_Proof2}
			V(\mathbf{X}^{\mathbf{L}}(\tau^*)) \le \psi(\mathbf{X}^{\mathbf{L}}(\tau^*)) - 2 \varepsilon. 
		\end{equation}
		Using a bivariate extension of \cite[Prop. 2.13]{Azcue2014} we obtain 
		
		\begin{equation} \label{eq_Proof_SubSolution_Contradiction_Proof3}
			\begin{aligned}
				e^{-q\tau^*}  & \psi \left(\mathbf{X}^{\mathbf{L}}(\tau^*)\right) - \psi(\tilde{\mathbf{x}}) \\
				& = \int_0^{\tau^*} e^{-qs} \mathcal{L}(\psi)\left(\mathbf{X}^{\mathbf{L}}(s-)\right) \diff s - \int_0^{\tau^*} e^{-qs} \diff \mathbf{L}(s) \\
				&\qquad + \int_0^{\tau^*} \left[1- \psi_{x_1}\left(\mathbf{X}^{\mathbf{L}}(s-)\right)\right] e^{-qs}\diff L_1^c(s)  + \int_0^{\tau^*} \left[1- \psi_{x_2}\left(\mathbf{X}^{\mathbf{L}}(s-)\right)\right] e^{-qs}\diff L_2^c(s) \\
				&\qquad  + \sum_{\substack {X_1(s)\neq X_1(s+) \\ s< \tau^*}} e^{-qs}\left(\int_0^{L_1(s+)-L_1(s)}\left[1- \psi_{x_1}\left( X_1^{\mathbf{L}}(s)-\alpha, X_2^{\mathbf{L}}(s)\right)\right]\diff \alpha \right)  \\
				&\qquad  + \sum_{\substack {X_2(s)\neq X_2(s+) \\ s< \tau^*}} e^{-qs}\left(\int_0^{L_2(s+)-L_2(s)}\left[1- \psi_{x_2}\left( X_1^{\mathbf{L}}(s), X_2^{\mathbf{L}}(s)-\alpha\right)\right]\diff \alpha \right) \\
				&\qquad   + \tilde{M}(\tau^*),
			\end{aligned}
		\end{equation}
		where 
		
		\begin{equation*} 
			\begin{aligned}
				\tilde{M}(t) &:= \sum_{\substack {\mathbf{X}^{\mathbf{L}}(s-)\neq \mathbf{X}^{\mathbf{L}}(s) \\ s\leq t}} e^{-qs}\left( u \left(\mathbf{X}^{\mathbf{L}}(s)\right)-u\left(\mathbf{X}^{\mathbf{L}}(s-)\right)\right) \\   
				& \qquad - \lambda \int_0^t e^{-qs} \int_0^{\infty} u\left(X_1^{\mathbf{L}}(s-)-b_1\alpha, X_2^{\mathbf{L}}(s-)-b_2\alpha\right) - u\left(\mathbf{X}^{\mathbf{L}}(s-)\right)\diff F(\alpha) \diff s
			\end{aligned}
		\end{equation*}
		defines a zero mean martingale. Fix $i=1,2$. If $\tilde{x}_i<0$,  then by construction of $h$ we have $X_i^{\mathbf{L}}(s) <0$ for any $s\in (0,\tau^*]$. Hence, by admissibility of the strategy, no dividends can be paid from branch $i$ and the respective integrals in \eqref{eq_Proof_SubSolution_Contradiction_Proof3} are zero. If on the other hand $\tilde{x}_i\geq 0$, then we may apply \eqref{eq_Proof_SubSolution_Construction1} or \eqref{eq_Proof_SubSolution_Construction2} to get 
		
		\begin{equation}  \label{eq_Proof_SubSolution_Contradiction_Proof3b}
		    e^{-q\tau^*}   \psi \left(\mathbf{X}^{\mathbf{L}}(\tau^*)\right) - \psi(\tilde{\mathbf{x}})  \le \int_0^{\tau^*} e^{-qs} \mathcal{L}(\psi)\left(\mathbf{X}^{\mathbf{L}}(s-)\right) \diff s - \int_0^{\tau^*} e^{-qs} \diff \mathbf{L}(s) + \tilde{M}(\tau^*).
		\end{equation}
		From \eqref{eq_Proof_SubSolution_Construction3}, \eqref{eq_Proof_M_compact_bounded}, and \eqref{eq_Proof_ThetaInequality}, we have 
		
		\begin{equation}
			\label{eq_Proof_SubSolution_Contradiction_Proof4}
			\begin{aligned}
				\int_0^{\tau^*} & e^{-qs} \mathcal{L}(\psi)\left(\mathbf{X}^{\mathbf{L}}(s-)\right) \diff s \\
				& = \int_0^{\tau^{\mathbf{L}} \wedge \underline{\tau}} e^{-qs} \mathcal{L}(\psi)\left(\mathbf{X}^{\mathbf{L}}(s-)\right) \diff s + \int_{\tau^{\mathbf{L}} \wedge \underline{\tau}}^{\tau^*} e^{-qs} \mathcal{L}(\psi)\left(\mathbf{X}^{\mathbf{L}}(s-)\right) \diff s  \\
				& \le - 2 \varepsilon q \int_0^{\tau^{\mathbf{L}} \wedge \underline{\tau}} e^{-qs} \diff s + C\theta  \\
				& \le  - 2 \varepsilon q \int_0^{\tau^{\mathbf{L}} \wedge \underline{\tau}} e^{-qs} \diff s + \frac{\varepsilon}{2}, 
			\end{aligned}
		\end{equation}
		where the inequality for the second integral holds, since $\mathbf{X}^{\mathbf{L}}(s-)$ is in the union of some compact set $M$ and $\mathbb{R}^2_{<0}$. This is because, due to admissibility, we can not force a branch to go negative by a dividend payment and because claims occur along the line $x_2 = \frac{b_2}{b_1}\cdot x_1$.  \\ 
		Now, the rest of the proof is completely similar to the proof of \cite[Prop. 3.1]{Azcue2014}: We use the Dynamic Programming Principle (Proposition \ref{Proposition_DPP}) together with Equations \eqref{eq_Proof_SubSolution_Contradiction_Proof2}, \eqref{eq_Proof_SubSolution_Contradiction_Proof3b} and \eqref{eq_Proof_SubSolution_Contradiction_Proof4} to obtain the desired inequality 
		
		\begin{align*}
			V(\tilde{\mathbf{x}}) &= \sup_{\mathbf{L} \in \Pi_{\tilde{\mathbf{x}}}} \mathbb{E}_{\tilde{\mathbf{x}}} \left[\int_{0}^{\tau^*} e^{-qs} \diff L_1(s) + \int_{0}^{\tau^*} e^{-qs} \diff L_2(s) + e^{-q\tau^*} V(\mathbf{X}^{\mathbf{L}}(\tau^*))\right] \\
			&\leq 	\psi(\tilde{\mathbf{x}}) - \varepsilon < \psi(\tilde{\mathbf{x}}). \qedhere
		\end{align*}
	\end{proof}
	
	\begin{remark}
		A comparison of our proof and the proof of Theorem 3.5 in \cite{Azcue2018} exhibits some flaws in the latter. The authors state that
		
		\begin{quotation}
			``$V(\overline{x})\leq \psi(\overline{x}) - 2\epsilon$ for $\overline{x}\in[-\infty,\overline{x}_0-h/2]\cup\{\overline{x}_0+h\}$''  
		\end{quotation}
		and later use the same stopping times $\overline{\tau}$, $\uline{\tau}$ and $\tau^*$ as we do. This however is not enough, as 
		
		\begin{equation*}
			\overline{X}(\tau^*) \in[-\infty,\overline{x}_0-h/2]\cup\{\overline{x}_0+h\}
		\end{equation*} does \emph{not} necessarily hold (see also \eqref{eq_Proof_SubSolution_Contradiction_Proof1}) and hence, the multivariate generalization of \cite[Eq. (3.20)]{Azcue2014} fails. \\
		Nevertheless, we emphasize that this does not affect the statement of \cite[Theorem 3.5]{Azcue2018} itself, as the proof may be fixed by adjusting the definition of $\psi_1$ as 
		
		\begin{equation*} 
			\psi_1(\mathbf{x}):=\psi_0(\mathbf{x}) + \left(\frac{\kappa}{(b_1^2+ b_2^2)\left(\frac{\tilde{x}_1}{b_1} \vee \frac{\tilde{x}_2}{b_2} \right)^2}\right)  \left((x_1-\tilde{x}_1)^2 + (x_2-\tilde{x}_2)^2\right)
		\end{equation*} 
		for any $\mathbf{x}=(x_1,x_2)\in\mathcal{S}$, in order to show that \eqref{eq_Proof_SubSolution_Construction4}	holds on the larger set $B^*$. 
	\end{remark}
	
	The following proposition is also called \emph{verification result}: 
	
	\begin{proposition} \label{Proposition_Verification}
		The optimal value function is the smallest viscosity solution $u$ of \eqref{eq_HJB} satisfying the growth conditions 
		
		\begin{equation} \tag{G1} \label{eq_GrowthCondition1}
			u(x_1,x_2) \leq K + (x_1 \vee 0) + (x_2 \vee 0) \text{ for some } K>0 \text{ and any }(x_1,x_2)\in\mathcal{S}
		\end{equation}
		and 
		
		\begin{equation} \tag{G2} \label{eq_GrowthCondition2}
			u(x_1,x_2) < u(x_1 + h, x_2), \qquad   u(x_1,x_2)  < u(x_1, x_2+h)
		\end{equation} for any $(x_1,x_2)\in\mathcal{S}$ and $h>0$.
	\end{proposition}

	\begin{proof}
	    Similar to \cite[Prop. 4.4]{Azcue2014} one can show that for arbitrary $\mathbf{L}\in\Pi_{\mathbf{x}}$ and any viscosity supersolution $\overline{u}$ of \eqref{eq_HJB} satisfying \eqref{eq_GrowthCondition1} and \eqref{eq_GrowthCondition2}, it holds that $V^{\mathbf{L}}(\mathbf{x})\leq \overline{u}(\mathbf{x})$. Together with Proposition \ref{Proposition_ViscositySolution} this implies the statement. 
	\end{proof}
	
	\section{Bang strategies} \label{Section_CandidateOptimalStrategy}
	In this section we choose a heuristic approach to define a class of strategies that appeal to be optimal. We show that in certain subcases, the optimal strategy indeed lies in this class and we reduce the control problem defined in \eqref{optimalvaluefct} to a univariate one.  \\
	The idea behind the strategies is as follows: If we face ruin at time $\tau^{\mathbf{L}}$ then it is desirable to have a relatively small surplus right before $\tau^{\mathbf{L}}$ because it implies that we paid more dividends before ruin. As we allow that one branch of the process becomes negative, all possible dividends can be paid from one branch (at every $t\geq 0$) to ensure that there is no capital ``wasted'' at the time of ruin. We thus consider strategies that pay dividends according to the following principles: 

	\begin{itemize}
		\item One branch follows some admissible, one-dimensional dividend strategy, 
		\item the other branch pays dividends as follows: \begin{enumerate}[label=(\roman*)]
			\item If the surplus is positive, the whole surplus is immediately paid as a lump sum.
			\item If the surplus is zero, all incoming premia are continuously paid as dividends.
			\item If the surplus is negative, no dividends are paid until the branch reaches zero again. 
		\end{enumerate}	
	\end{itemize}
	Inspired by the well-known \emph{Bang-bang controls} (see e.g. \cite[Sec. 6.5]{Rolewicz1987}) we call these strategies \emph{bang strategies} as they always pay the maximum dividends possible from one branch. 
	Let $\kappa \in\{1,2\}$ and fix $\iota = 3-\kappa$ such that in particular $(\kappa,\iota) = (1,2)$ or $(2,1)$. Formally, for any $\mathbf{x} \in\mathcal{S}$ we set
	
	\begin{align*} 
		\Pi^{*\kappa}_{\mathbf{x}} &:= \Big\{\mathbf{L}=(L_1,L_2) : ~L_\kappa \text{ is admissible for }X_\kappa(t), \\
		& \qquad \ \qquad  L_\iota(t) = (x_\iota \vee 0) + c_\iota\cdot \int_0^{t} \mathds{1}_{\{\overline{X}_\iota(s)\geq 0\}} \cdot \mathds{1}_{\{\overline{X}_\iota(s) = X_\iota(s)\}} \diff s \Big\}  \nonumber 
	\end{align*}
	where $\overline{X}_\iota(s) := \sup_{0\leq r \leq s} X_\iota(r)$ denotes the running supremum of $X_\iota$. We define the class of bang strategies as
	
	\begin{equation}
		\Pi^*_{\mathbf{x}} := \Pi^{*1}_{\mathbf{x}} \cup \Pi^{*2}_{\mathbf{x}},
	\end{equation}
	where by construction any strategy in $\Pi^*_{\mathbf{x}}$ is admissible.
	Clearly, for any $\mathbf{L}^{1}\in\Pi^{*1}_{\mathbf{x}}$ the ruin time $\tau^{\mathbf{L}^{1}}$ is equal to the ruin time of the first branch, denoted by $\tau^{L_1}$, whereas for any $\mathbf{L}^{2}\in\Pi^{*2}_{\mathbf{x}}$ the ruin time $\tau^{\mathbf{L}^{2}}$ is equal to the ruin time of the second branch $\tau^{L_2}$. This is, because any occurring claim affects both branches of our risk process and by construction of $\mathbf{L}^{1}$, $\mathbf{L}^{2}$ any occurring claim leads to a negative surplus in the second or first branch, respectively. Further, for any univariate admissible strategy $L$ defined on branch $i$ it holds that 
	
	\begin{equation} \label{eq_bangStrategiesPayMost}
	L_i^*(t) := (x_i \vee 0) + c_i\cdot \int_0^{t} \mathds{1}_{\{\overline{X}_i(s)\geq 0\}} \cdot \mathds{1}_{\{\overline{X}_i(s)= X_i(s)\}} \diff s \geq L(t). 
	\end{equation} 
	In the following, to facilitate reading, we are going to use the expression ``strategy of type $\mathbf{L}^{*1}$'' to refer to a strategy in $\Pi_{\mathbf{x}}^{*1}$ and similarly for $\Pi_{\mathbf{x}}^{*2}$. \\ 
	Our next goal is to show that the optimal strategy for the control problem \eqref{optimalvaluefct} lies in $\Pi_{\mathbf{x}}^*$. 
	We start by proving that strategies of type $\mathbf{L}^{*1}$ and $\mathbf{L}^{*2}$ are the optimal choice for certain subsets of all admissible strategies. To do this, we define the sets  
	
	\begin{align*}
		\mathcal{D}_1&:= \{\mathbf{x} \in\mathcal{S}: ~(b_2/b_1)x_1\geq x_2\}, &		\mathcal{D}_2&:= \{\mathbf{x} \in\mathcal{S}: ~(b_2/b_1)x_1\leq x_2\}.
	\end{align*}
	By construction, the controlled process $\mathbf{X}^{\mathbf{L}}(t)$ can neither exit $\mathcal{D}_1^\circ$ into $\mathcal{D}_2$ nor exit $\mathcal{D}_2^\circ$ into $\mathcal{D}_1$ by a claim. Such a change can only happen by two events: 
	
    \begin{enumerate}[label=(\roman*)]
		\item The process deterministically creeps from $\mathcal{D}_2$ into $\mathcal{D}_1$, induced by the collected premia and by \eqref{eq_Assumption_c1b1c2b2}, or 
		\item the process is forced to change from one set to the other by dividend payments (continuously or by a lump sum).
	\end{enumerate} 
	
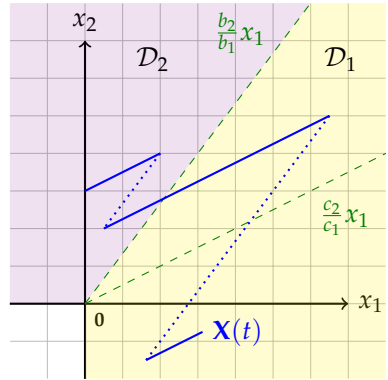
\begin{figure}[ht]
\centering
\begin{tikzpicture}[x=.5cm, y=.5cm, smooth]
   \draw [color=gray!50] [step=5mm] (-1.99,-1.99) grid (7.99,7.99);
   
   \draw[->,thick] (-2,0) -- (7,0) node[right] {$x_1$};
   \draw[->,thick] (0,-2) -- (0,7) node[above] {$x_2$};
   \node[below right]{$\scriptstyle\mathbf{0}$};
   
   \fill[color=violet, opacity = 0.12] (0,0) -- (-2,0) -- (-2,8) -- (6,8) -- (0,0);
   \fill[color=yellow, opacity = 0.18] (0,0) -- (0,-2) -- (8,-2) -- (8,8) -- (6,8) -- (0,0);
   
   	\draw[style= dashed, color=green!50!black] (0,0) --(6,8)node[near end, above=0.6cm, right=-0.65cm]{\color{green!50!black} $\frac{b_2}{b_1}x_1$};
	\draw[style= dashed, color=green!50!black] (0,0) -- (8,4)node[near end, below right]{\color{green!50!black} $\frac{c_2}{c_1}x_1$};
	\node at (6.8,6.5) {$\mathcal{D}_1$};
	\node at (1.8,6.5) {$\mathcal{D}_2$};
	
	\draw[thick, color=blue] (0,3) -- (2,4);
	\draw[thick, style=dotted, color=blue] (2,4) -- (0.5,2);
	\draw[thick, color=blue] (0.5,2) -- (6.5,5);
	\draw[thick, style=dotted, color=blue] (6.5,5) -- (1.625,-1.5);
	\draw[thick, color=blue] (1.625,-1.5) -- (3.125,-0.75) node[right]{\color{blue} $\mathbf{X}(t)$};
\end{tikzpicture}
\label{fig_D1D2}
\caption{Visualization of the sets $\mathcal{D}_1$, $\mathcal{D}_2$ and a sample path of $\mathbf{X}$.}
\end{figure}

	Note that event (i) can be interpreted as a special case of event (ii) since it corresponds to paying no dividends at this particular time. Hence, changes between the sets are determined by the dividend policy. \\ 
	For $\mathbf{x}\in \mathcal{D}_i$, $i=1,2$, let $\Pi^i_{\mathbf{x}}$ be the set of all admissible strategies $\mathbf{L}$ which ensure that $\mathbf{X}^{\mathbf{L}}$ stays in $\mathcal{D}_i$ until ruin. By the previous reasoning the sets $\Pi^1_{\mathbf{x}},~\Pi^2_{\mathbf{x}}$ are well-defined.
	
	\begin{Theorem} \label{Theorem_OptimalityIndicator_SpecialStrategy}
		The optimal strategy in $\Pi_{\mathbf{x}}^1$ is of type $\mathbf{L}^{*1}$, whereas the optimal strategy in $\Pi_{\mathbf{x}}^2$ is of type $\mathbf{L}^{*2}$.
	\end{Theorem}
	
	\begin{proof} We show the statement for $\Pi^1_{\mathbf{x}}$ as the proof for the second case is completely similar. 
		Let $\mathbf{L}^0 = (L_1^0,L_2^0)$ be any strategy in $\Pi^1_{\mathbf{x}}$. Define another strategy $\mathbf{L}^1$ as 
		
		\begin{equation*} 	
			\mathbf{L}^1 := (L_1^1, L_2^1) := (L_1^0, L_2^*),	
		\end{equation*} with $L_2^*$ as in \eqref{eq_bangStrategiesPayMost}. Then $\mathbf{L}^1\in\Pi^1_{\mathbf{x}}$ and it holds that $L_2^1(t) \geq L_2^0(t)$ for any $0\leq t\leq \tau^{\mathbf{L}^1}$. Hence, 
		
		\begin{equation*} 
		\mathbf{L}^1(t) \geq \mathbf{L}^0(t)  \qquad \text{ for any } 0<t\leq \tau^{\mathbf{L}^1} \leq \tau^{\mathbf{L}^0}.	
		\end{equation*}
		Moreover, by definition of $\mathbf{L}^0$ and $\mathbf{L}^1$ we have for all $0\leq t \leq \tau^{\mathbf{L}_0}$
		
		\begin{align*} 
			 \frac{b_2}{b_1} \cdot X_1^{\mathbf{L}^1}(t) = \frac{b_2}{b_1} \cdot X_1^{\mathbf{L}^0}(t) \geq X_2^{\mathbf{L}^0}(t) \geq X_2^{\mathbf{L}^1}(t),	
		\end{align*}
		which directly implies that $\tau^{\mathbf{L}^1} = \tau^{\mathbf{L}^0} = \tau^{L^0_1}$. Hence, $V^{\mathbf{L}^1}(\mathbf{x})\geq V^{\mathbf{L}^0}(\mathbf{x})$ and the statement follows as $\mathbf{L}^0\in\Pi^1_{\mathbf{x}}$ was arbitrary. 
	\end{proof}
	
	\begin{remark}
		Note that in contrast to the results in \cite[Section 4.3]{Azcue2018}, Theorem \ref{Theorem_OptimalityIndicator_SpecialStrategy} implies that a strategy that stays in $\mathcal{D}_1\cap \mathcal{D}_2 = \{(x_1, \frac{b_1}{b_2}x_1), ~x_1\geq 0\}$ until ruin can \emph{never} be optimal in our setting.
	\end{remark}
Next, we show that under some additional assumptions a strategy of type $\mathbf{L}^{*1}$ is indeed optimal:

	\begin{Theorem} \label{Theorem_BetterStrategyConstruction}
 		Let $b_1 \leq b_2$ and let $\mathbf{x}$ in $\mathcal{D}_1$. Then for any strategy $\mathbf{L}^0 = (L_1^0,L_2^0)\in\Pi_{\mathbf{x}}$ there exists a strategy $\mathbf{L}^1\in\Pi_{\mathbf{x}}^1$ such that $V^{\mathbf{L}^0}(\mathbf{x}) \leq V^{\mathbf{L}^1}(\mathbf{x})$.
	\end{Theorem}
	
	Before we begin with the proof of Theorem \ref{Theorem_BetterStrategyConstruction}, we prove a preparatory lemma: 
	
	\begin{lemma}
		It holds that
		
		\begin{equation} \label{eq_Proof_D1Strategy_5}
			c_2 \cdot \left(\int_0^t \mathds{1}_{\{\overline{X}_2(s) = X_2(s)\}}\diff s - t \right) \geq 	\frac{b_2}{b_1} \cdot c_1 \cdot \left(\int_0^t \mathds{1}_{\{\overline{X}_1(s) = X_1(s)\}}\diff s - t \right).
		\end{equation}	
	\end{lemma}
	\begin{proof}
Let 

\begin{equation} \label{eq_Proof_D1Strategy_Definition_Oi}
O_i(t) := \overline{X}_i(t) - X_i(t) = \sum_{n=1}^{N(t)} b_i U_n - c_i \int_0^t \mathds{1}_{\{\overline{X}_i(s) \neq X_i(s)\}}\diff s
\end{equation} be the offset of branch $i$ from its running supremum. Note that the offset is independent of the initial value $X_i(0)=x_i$. Hence, w.l.o.g. we may assume $x_1=x_2= 0$. It holds that 

\begin{equation} \label{eq_Proof_D1Strategy_InequalityOi}
	\begin{aligned}
			\frac{b_1}{b_2} O_2(t) &= \frac{b_1}{b_2} \overline{X}_2(t) - \frac{b_1}{b_2} X_2(t) \\
			&= \sup_{0\leq s\leq t} \left( \frac{b_1}{b_2}c_2s - \sum_{n=1}^{N(s)} b_1 U_n\right) + \left(c_1- \frac{b_1}{b_2}c_2 \right)t - X_1(t) \\
			&\geq \sup_{0\leq s\leq t} \left( \frac{b_1}{b_2}c_2s + \left(c_1- \frac{b_1}{b_2}c_2 \right)s - \sum_{n=1}^{N(s)} b_1 U_n\right) - X_1(t) \\
			&= O_1(t), 	
	\end{aligned}
\end{equation}
where for the inequality we used that $(c_1-\frac{b_1}{b_2}c_2)t$ is monotonically increasing by \eqref{eq_Assumption_c1b1c2b2}. Equations \eqref{eq_Proof_D1Strategy_Definition_Oi} and \eqref{eq_Proof_D1Strategy_InequalityOi} imply 

	\begin{align*}
			c_2 \cdot \int_0^t \mathds{1}_{\{\overline{X}_2(s) \neq X_2(s)\}}\diff s &= \sum_{n=1}^{N(t)} b_2 U_n - O_2(t) 
			= \frac{b_2}{b_1}\cdot\left(\sum_{n=1}^{N(t)} b_1 U_n - \frac{b_1}{b_2} \cdot O_2(t)  \right) \\
			& \leq  \frac{b_2}{b_1}\cdot \left(\sum_{n=1}^{N(t)} b_1 U_n - O_1(t)  \right) 
			 = \frac{b_2}{b_1} \cdot c_1 \cdot\int_0^t \mathds{1}_{\{\overline{X}_1(s) \neq X_1(s)\}}\diff s,
		\end{align*} 
which finishes the proof, as 

\begin{equation*}
	\int_0^t\mathds{1}_{\{\overline{X}_i(s) \neq X_i(s)\}}\diff s = t-\int_0^t\mathds{1}_{\{\overline{X}_i(s) = X_i(s)\}}\diff s. \qedhere
\end{equation*}
\end{proof}
	
Now we are ready to prove Theorem \ref{Theorem_BetterStrategyConstruction}: 

\begin{proof}[Proof of Theorem \ref{Theorem_BetterStrategyConstruction}]
	We start by showing the statement for the slightly stronger assumption that $\mathbf{x}\in\mathcal{D}_1\cap \mathbb{R}^2_{\geq 0}$. Let $\mathbf{L}^0=(L^0_1,L^0_2)\in\Pi_\mathbf{x}$ be arbitrary. For any $t\geq 0$ we define a strategy $\mathbf{L}^1=(L_1^1,L_2^1)$ by
	
		\begin{align}
			\label{eq_Proof_D1Strategy_1} L_1^1(t) &:= \min\left\{L_1^0(t), x_1 -\frac{b_1}{b_2}x_2 + \left(c_1- \frac{b_1}{b_2}\cdot c_2\right) t  + \frac{b_1}{b_2} \cdot L_2^0(t) \right\}, \\
			\label{eq_Proof_D1Strategy_2}	L_2^1(t) &:= L^*_2(t) .
		\end{align}
		Note that $L_1^1$ is increasing by \eqref{eq_Assumption_c1b1c2b2} and nonnegative as $\mathbf{x}\in\mathcal{D}_1$. Moreover, admissibility of $L_1^1$ follows directly from admissibility of $\mathbf{L}^0$. By definition, $L_2^1$ is admissible and it is clear that $\mathbf{X}^{\mathbf{L}^1}(t)\in\mathcal{D}_1$ for all $t\geq 0$. Hence, it remains to show that indeed $V^{\mathbf{L}^1}(\mathbf{x}) \geq  V^{\mathbf{L}^0}(\mathbf{x})$. To accomplish this, we show 
		
		\begin{equation} \label{eq_Proof_D1Strategy_3}
			\tau^{\mathbf{L}^0} \leq \tau^{\mathbf{L}^1} \qquad 	\text{ and } \qquad L_1^1(t) + L_2^2(t)  \geq L_1^0(t) + L_2^0(t), \qquad \forall t\geq 0.	
		\end{equation}
		By \eqref{eq_Proof_D1Strategy_1} it holds that
		
		\begin{equation} \label{eq_Proof_D1Strategy_NeededInRemark}
		\begin{aligned}
			X_1^{\mathbf{L}^1}(t) &= x_1 + c_1 t -\sum_{n=1}^{N(t)} b_1 U_n - L_1^1(t) \\ 
			&= \max\left\{x_1 + c_1 t -\sum_{n=1}^{N(t)} b_1 U_n - L_1^0(t), \frac{b_1}{b_2}\left( x_2 + c_2 t -\sum_{n=1}^{N(t)} b_2 U_n - L_2^0(t)\right)\right\} \\
			&= \max\left\{ X_1^{\mathbf{L}^0}(t), \frac{b_1}{b_2}\cdot X_2^{\mathbf{L}^0}(t)\right\} 
		\end{aligned}
		\end{equation}
		and hence $\tau^{\mathbf{L}^0}\leq \tau^{\mathbf{L}^1}$. \\ 
		To show the second inequality in \eqref{eq_Proof_D1Strategy_3} we consider two separate cases. Let first $\mathbf{X}^{\mathbf{L}^0}(t)\in\mathcal{D}_1$. Then $\frac{b_2}{b_1} X_1^{\mathbf{L}^0}(t) \geq X_2^{\mathbf{L}^0}(t)$, which is equivalent to 
		
		\begin{align*}
			L_1^0(t) & \leq x_1 - \frac{b_1}{b_2}\cdot x_2 + \left(c_1- \frac{b_1}{b_2}\cdot c_2\right) t  + \frac{b_1}{b_2} \cdot L_2^0(t).
		\end{align*}
		Hence, in this case the minimum in \eqref{eq_Proof_D1Strategy_1} is attained by $L_1^0(t)$. Moreover, by definition of $L_2^*(t)$ we have $L_2^1(t)\geq L_2^0(t)$ for any $t\geq 0$, and thus 
		
		\begin{equation*}
			L_1^1(t) + L_2^2(t)  \geq L_1^0(t) + L_2^0(t).
		\end{equation*}
		If otherwhise $\mathbf{X}^{\mathbf{L}^0}(t)\in\mathcal{D}_2$, then analogously we have  
		
		\begin{equation} \label{eq_Proof_D1Strategy_4}
			L_2^0(t) \leq x_2 - \frac{b_2}{b_1}\cdot x_1 + \left(c_2- \frac{b_2}{b_1}\cdot c_1\right) t  + \frac{b_2}{b_1} \cdot L_1^0(t),
		\end{equation}  
		and the minimum in \eqref{eq_Proof_D1Strategy_1} is attained by the second term. We use \eqref{eq_Proof_D1Strategy_4} to show
		
		\begin{equation} \label{eq_Proof_D1Strategy_8}
			\begin{aligned}
				L_1^0(t)  + & \left(1-\frac{b_1}{b_2}\right) L_2^0(t)  - \left( x_1 -\frac{b_1}{b_2}x_2 + \left(c_1- \frac{b_1}{b_2}\cdot c_2\right) t\right) \\ 
				&\leq L_1^0(t) + \left(1-\frac{b_1}{b_2}\right)\cdot \left(x_2 - \frac{b_2}{b_1}\cdot x_1 + \left(c_2- \frac{b_2}{b_1}\cdot c_1\right) t  + \frac{b_2}{b_1} \cdot L_1^0(t) \right) \\
				&\qquad - \left( x_1 -\frac{b_1}{b_2}x_2 + \left(c_1- \frac{b_1}{b_2}\cdot c_2\right) t\right) \\
				&= \frac{b_2}{b_1} \left( L_1^0(t) - x_1 + \frac{b_1}{b_2} x_2 -c_1t + \frac{b_1}{b_2}\cdot c_2 t  \right).
			\end{aligned}
		\end{equation}
		Note that, here, we used $b_1\leq b_2$ to ensure that $1-\frac{b_1}{b_2}\geq 0$. \\ 
		Moreover, it holds that (see \eqref{eq_bangStrategiesPayMost})
		
		\begin{equation*}
			L_1^0(t) \leq L^*_1(t) =  x_1  + c_1\cdot \int_0^{t} \mathds{1}_{\{\overline{X}_1(s)= X_1(s)\}} \diff s.
		\end{equation*}
		With this, and by \eqref{eq_Proof_D1Strategy_5} we have 
	
		\begin{equation} \label{eq_Proof_D1Strategy_9}
			\begin{aligned}
				\frac{b_2}{b_1} \cdot &\left( L_1^0(t) -x_1 +\frac{b_1}{b_2} x_2 - c_1 t + \frac{b_1}{b_2}c_2 t\right) \\
				&\leq \frac{b_2}{b_1} \cdot \left(  x_1  + c_1\cdot \int_0^{t} \mathds{1}_{\{\overline{X}_1(s)= X_1(s)\}} \diff s -x_1 +\frac{b_1}{b_2} x_2 - c_1 t + \frac{b_1}{b_2}c_2 t\right) \\
				&= \frac{b_2}{b_1} \cdot c_1\cdot \left( \int_0^{t} \mathds{1}_{\{\overline{X}_1(s) = X_1(s)\}} \diff s -t \right) + x_2 + c_2 t \\ 
				&\leq  x_2  + c_2\cdot \int_0^{t} \mathds{1}_{\{\overline{X}_2(s) = X_2(s)\}} \diff s = L_2^*(t) = L_2^1(t),
			\end{aligned} 
		\end{equation} since $x_2\geq 0$.
		We combine \eqref{eq_Proof_D1Strategy_8} and \eqref{eq_Proof_D1Strategy_9} and obtain 
	
		\begin{align*}
			L_2^1(t) & \geq L_1^0(t) +\left(1-\frac{b_1}{b_2}\right) L_2^0(t)  - \left( x_1 -\frac{b_1}{b_2}x_2 + \left(c_1- \frac{b_1}{b_2}\cdot c_2\right) t\right) \\
			&= L_1^0(t) + L_2^0(t)  - \left( x_1 -\frac{b_1}{b_2}x_2 + \left(c_1- \frac{b_1}{b_2}\cdot c_2\right) t  + \frac{b_1}{b_2} \cdot L_2^0(t)\right) .
		\end{align*}
		Finally, as the minimum in \eqref{eq_Proof_D1Strategy_1} is attained by the second term, this implies that 
	
		\begin{align*}
			L_1^1(t) + L_2^1(t) &= x_1 -\frac{b_1}{b_2}x_2 + \left(c_1- \frac{b_1}{b_2}\cdot c_2\right) t  + \frac{b_1}{b_2} \cdot L_2^0(t) + L_2^1(t) \\
			& \geq L_1^0(t) + L_2^0(t)
		\end{align*}
		and hence the proof for $\mathbf{x}\in\mathcal{D}_1\cap\mathbb{R}^2_{\geq 0}$ is finished. Note that the additional restriction $x_2\geq 0$ can be dropped, since in the case of $x_2<0$, by admissibility, no dividends are paid until $X_2(t)$ reaches zero. Hence, any admissible strategy on the second branch coincides with the strategy $L_2^*$. Once the process $X_2$ reaches zero, we may apply the restricted result.
		\end{proof}
	
	\begin{remark} \label{Remark_WhyNotExtendable}
		Note that the assumption $b_1\leq b_2$ in Theorem \ref{Theorem_BetterStrategyConstruction} does indeed impose a restriction to our model. As we assumed \eqref{eq_Assumption_c1b1c2b2} throughout the whole paper, we can \emph{not} simply exchange branches in order to obtain  the case $b_2<b_1$. \\
		Unfortunately, even though $b_1\leq b_2$ is only used in \eqref{eq_Proof_D1Strategy_8}, it is not possible to adapt the proof to obtain a similar result neither for the case $b_2<b_1$, nor for $\mathbf{x}\in\mathcal{D}_2$ and in the following we heuristically explain why. \\
		The general idea of the proof is, given the arbitrary strategy $\mathbf{L}^0$, to construct another strategy $\mathbf{L}^1\in \Pi_\mathbf{x}^{1}$ that fulfills $\tau^{\mathbf{L}^1}\geq \tau^{\mathbf{L}^0}$ almost surely. (Actually, \eqref{eq_Proof_D1Strategy_1} and \eqref{eq_Proof_D1Strategy_2} ensure that \eqref{eq_Proof_D1Strategy_NeededInRemark} implies even $\tau^{\mathbf{L}^1}=\tau^{\mathbf{L}^0}$.) \\
		This construction relies heavily on the assumptions \eqref{eq_Assumption_c1b1c2b2} and $\mathbf{x}\in\mathcal{D}_1$. Otherwise $L_1^1(t)$ would not be admissible, as $c_1-\frac{b_1}{b_2}c_2 <0$ or $x_1-\frac{b_1}{b_2}x_2 <0$, respectively. Moreover, in order to fulfill the second inequality in \eqref{eq_Proof_D1Strategy_3}, $b_1\leq b_2$ is essential. We illustrate this using a counterexample: Assume $b_1>b_2$ and fix a strategy $\mathbf{L}^0$ such that 
		
		\begin{itemize} 
		\item at $t=0$, the first branch pays the whole inital capital $x_1$ as lump sum, while 
		\item the second branch does not make any dividend payments at $t=0$.
		\end{itemize} 
		Let $x_1,x_2>0$. Then, by \eqref{eq_Proof_D1Strategy_1} and \eqref{eq_Proof_D1Strategy_2} we have
		\begin{align*}
		    L_1^1(0+) + L_2^1(0+) &= x_1 -\frac{b_1}{b_2}x_2 + \frac{b_1}{b_2} \cdot L_2^0(0+) + x_2  \\
		    &= x_1 - \left(1-\frac{b_1}{b_2}\right) x_2 
		    < x_1  = L_1^0(0+) + L_2^0(0+),
	    \end{align*}  where $L_i^j(0+):=\lim_{t\downarrow 0} L_i^j(t)$, $j=0,1$, $i=1,2$.
		Hence, if $b_1>b_2$, then the construction \eqref{eq_Proof_D1Strategy_1} - which ensures that $\tau^{\mathbf{L}^0}\leq \tau^{\mathbf{L}^1}$ - does not allow for higher dividend payments in general.
	\end{remark}
	
	The following Corollary summarizes the important consequences of Theorems \ref{Theorem_OptimalityIndicator_SpecialStrategy} and   \ref{Theorem_BetterStrategyConstruction}. 
	
	\begin{corollary} \label{Corollary_OptimalStrategyUnderAssumptions}
		Let $b_1\leq b_2$. Then for any $\mathbf{x} \in \mathcal{D}_1$, the optimal strategy is of type $\mathbf{L}^{*1}$. \\ 
		If $\mathbf{x}\in\mathcal{D}_2$ then either the optimal strategy is of type $\mathbf{L}^{*2}$, or the optimal strategy's controlled process enters $\mathcal{D}_1$ with positive probability, in which case it is optimal to continue with a strategy of type $\mathbf{L}^{*1}$.
	\end{corollary}
	
	As explained in Remark \ref{Remark_WhyNotExtendable} our approach for proving the optimality of bang strategies fails in the case $b_2<b_1$ as well as if $b_1\leq b_2, ~\mathbf{x}\in\mathcal{D}_2$. Even though in the latter case Corollary \ref{Corollary_OptimalStrategyUnderAssumptions} does make a statement on the two possible behaviors of the optimal strategy, there remains an open question: If the optimal strategy's controlled process enters $\mathcal{D}_1$ with positive probability, then the corollary does not make any statement on how the optimal strategy behaves until this event. \\
	As both of these problems could possibly be solved using other techniques, we leave them as open questions for future research.
	
	\section*{Value functions of bang strategies} 
	Next, we study properties of the value functions of strategies of type $\mathbf{L}^{*1}$ and $\mathbf{L}^{*2}$. By symmetry, the classes are completely similar and we thus focus on the former. \\
	For any $\mathbf{x}=(x_1,x_2)\in\mathcal{S}$ the value function $V^{\mathbf{L}^{*1}}(x_1,x_2)$ of the strategy $\mathbf{L}^{*	1} = (L_1,L_2^*)$ can be expressed as

	\begin{equation} \label{eq_SpecialStrategyValueFunc}
		\begin{aligned}
			&V^{\mathbf{L}^{*1}}  (x_1,x_2) 
		\\	&= (x_2 \vee 0) + \mathbb{E}_{x_1,x_2}\left[\int_0^{\tau^{L_1}} e^{-qs} \diff L_1(s) + c_2\cdot \int_0^{\tau^{L_1}}\mathds{1}_{\{\overline{X}_2(s)\geq 0\}} \cdot \mathds{1}_{\{\overline{X}_2(s) = X_2(s)\}} \cdot e^{-qs} \diff s \right].
		\end{aligned}
	\end{equation}
	Note that if $x_2\geq 0$, then $\mathds{1}_{\{\overline{X}_2(s)\geq 0\}} \equiv 1$ and hence, the expression can be dropped in the second integral of \eqref{eq_SpecialStrategyValueFunc}. If the initial capital $x_1$ of branch one is negative, then we may characterize the value of strategy $\mathbf{L}^{*1}$ explicitly in terms of $V^{\mathbf{L}^{*1}}(0,0)$: 

	\begin{lemma} \label{Lemma_SpecialStrategyNeverGoSet}
		For any $x_1<0$, $x_2\geq 0$ and any admissible strategy $L_1$ on branch one it holds that
		
		\begin{equation} \label{eq_ExplicitRepresentationVForX1Negative}
			V^{\mathbf{L}^{*1}}(x_1,x_2) = x_2 + \frac{c_2}{q+\lambda} + \left(V^{\mathbf{L}^{*1}}(0,0) -\frac{c_2}{q+\lambda} \right)\cdot  e^{(q+\lambda)\cdot\frac{x_1}{c_1}}.
		\end{equation}
		In particular, as $x_1\downarrow -\infty$, $V^{\mathbf{L}^{*1}}$ is exponentially decreasing and  
		
		\begin{equation}
			\lim_{x_1\to-\infty} V^{\mathbf{L}^{*1}}(x_1,x_2) = x_2 + \frac{c_2}{q+\lambda}.
		\end{equation}
	The result holds analogously for $x_1\geq 0, ~x_2<0$ and strategies of type $\mathbf{L}^{*2}$.  
	\end{lemma}

	\begin{proof}
		Note that by construction, branch two pays $x_2$ as a lump sum at the beginning and continues paying constantly $c_2\diff t$ as dividends, whenever the controlled process is positive. This construction ensures that $X_2^{L^*_2}(t)\leq 0$. Consequently, if the first claim at time $\tau_1$ happens before branch one gets positive, i.e., if $\tau_1\leq t_0 :=  -\tfrac{x_1}{c_1}$, then the value of the strategy is exactly the discounted value of the dividends paid from branch two until the first claim. If the first claim happens after $t_0$, then the value of the strategy at $t_0$ is exactly $e^{-qt_0}\cdot V^{\mathbf{L}^{*1}}(0,0)$ plus the discounted dividends from branch two up to time $t_0$. As $\tau_1\sim \operatorname{Exp}(\lambda)$ and due to the lack of memory of the exponential distribution, we conclude that for any $x_1<0$, $x_2\geq 0$ it holds that
	
		\begin{align*}
			V^{\mathbf{L}^{*1}}(x_1,x_2) &= x_2 + \mathbb{E}\left[\mathds{1}_{\{\tau_1\leq t_0\}} V^{\mathbf{L}^{*1}}(x_1,0) + \mathds{1}_{\{\tau_1> t_0\}} V^{\mathbf{L}^{*1}}(x_1,0)   \right]\\
			&= x_2 + \int_0^{t_0} \lambda e^{-\lambda t}\int_0^{t} e^{-qs} c_2 \diff s\diff t \\
			& \qquad + \mathbb{P}(\tau_1> t_0)\cdot \left( e^{-q \cdot t_0}\cdot V^{\mathbf{L}^{*1}}(0,0) + \int_0^{t_0} e^{-qs}c_2\diff s \right) \\
			&= x_2 + \frac{c_2}{q}\cdot\left( 1- \frac{\lambda}{q+\lambda}\right)  + \left(V^{\mathbf{L}^{*1}}(0,0) -\frac{c_2}{q}\cdot\left(1- \frac{\lambda}{\lambda+q} \right) \right)\cdot  e^{(q+\lambda)\cdot\frac{x_1}{c_1}}. \qedhere
		\end{align*}
	\end{proof}
	
	\begin{remark} \label{Remark_NeverGoSet}
		The asymptotics as $x_1\downarrow -\infty$ indicate that $\mathbf{L}^{*1}$ is \emph{not} the optimal strategy on $\mathbb{R}_{<0}\times \mathbb{R}_{\geq 0}$ because as long as the surplus of branch one is negative, $\mathbf{L}^{*1}$ collapses to the trivial strategy, that always pays the most dividend possible. However, if the initial capital $\mathbf{x}$ is in $\mathbb{R}^2_{\geq 0}$, then the process $\mathbf{X}^{\mathbf{L}^{*1}}$ never actually enters this set, as by construction of $\mathbf{L}^{*1}$ the occurrence of $X_1^{\mathbf{L}^{*1}}(t)<0$ implies ruin. 
	\end{remark}
	
	The next Lemma gives sufficient conditions on when $\mathbf{L}^{*1}$ is preferable over $\mathbf{L}^{*2}$ and vice versa. 

	\begin{lemma} \label{Lemma_V*1vsV*2}
		Let $x_1\geq 0, ~x_2\in\mathbb{R}$. If $V^{\mathbf{L}^{*1}}(x_1,x_2)\geq V^{\mathbf{L}^{*2}}(x_1,x_2)$, then we have 
	
		\begin{equation*}
			V^{\mathbf{L}^{*1}}(x_1+h,x_2) \geq V^{\mathbf{L}^{*2}}(x_1+h,x_2) \quad \text{ for all } h>0.
		\end{equation*}
		If otherwise $x_2\geq 0$, $x_1\in\mathbb{R}$ and $V^{\mathbf{L}^{*2}}(x_1,x_2)\geq V^{\mathbf{L}^{*1}}(x_1,x_2)$, then 
	
		\begin{equation*}
			V^{\mathbf{L}^{*2}}(x_1,x_2+h) \geq V^{\mathbf{L}^{*1}}(x_1,x_2+h) \quad \text{ for all } h>0.
		\end{equation*}
	\end{lemma}
	
	\begin{proof}
		Assume that $V^{\mathbf{L}^{*1}}(x_1,x_2)\geq V^{\mathbf{L}^{*2}}(x_1,x_2)$ for some $x_1\geq 0,~ x_2\in\mathbb{R}$. Then obviously $V^{\mathbf{L}^{*1}}(x_1+h,x_2)\geq V^{\mathbf{L}^{*1}}(x_1,x_2)+h$ as branch one can always pay a lump sum of size $h$. We conclude by construction of $\mathbf{L}^{*2}$ that 
	
		\begin{align*}
			V^{\mathbf{L}^{*1}}(x_1+h,x_2) &\geq h+V^{\mathbf{L}^{*1}}(x_1,x_2) =  V^{\mathbf{L}^{*2}}(x_1+h,x_2).
		\end{align*}
		The proof of the second case is completely analog.
	\end{proof}

\section*{The one-dimensional optimization problem}
As the last step in this section, we formulate the one-dimensional optimization problem that arises for bang strategies. Let $\Pi_{x_i}$, $i=1,2$, be the set of all admissible strategies (in the classical univariate sense, see \cite[Chapter 1.2]{Azcue2014}) acting on branch $i$ with initial capital $x_i$. For technical reasons, we allow $x_i$ to be negative and extend the definition, such that strategies are admissible if they do not pay any dividends, while branch $i$ is negative. Define

\begin{align} \label{eq_SpecialStrategyOptimizationProblem1}
	&V^{*1} (x_1,x_2) := \\ &(x_2 \vee 0) + \sup_{L_1\in\Pi_{x_1}}  \mathbb{E}_{x_1,x_2}\left[\int_0^{\tau^{L_1}} e^{-qs} \diff L_1(s) + c_2\cdot \int_0^{\tau^{L_1}}\mathds{1}_{\{\overline{X}_2(s)\geq 0\}} \cdot \mathds{1}_{\{\overline{X}_2(s) = X_2(s)\}} \cdot e^{-qs} \diff s \right] \nonumber
\end{align}
and 

\begin{align} \label{eq_SpecialStrategyOptimizationProblem2}
			&V^{*2} (x_1,x_2) := \\ &(x_1 \vee 0) + \sup_{L_2\in\Pi_{x_2}}  \mathbb{E}_{x_1,x_2}\left[ c_1\cdot \int_0^{\tau^{L_2}}\mathds{1}_{\{\overline{X}_1(s)\geq 0\}} \cdot \mathds{1}_{\{\overline{X}_1(s) = X_1(s)\}} \cdot e^{-qs} \diff s+\int_0^{\tau^{L_2}} e^{-qs} \diff L_2(s)  \right]. \nonumber
\end{align}
	Corollary \ref{Corollary_OptimalStrategyUnderAssumptions} implies, that if $b_1\leq b_2$ and $\mathbf{x}\in\mathcal{D}_1$ then the solution of \eqref{eq_SpecialStrategyOptimizationProblem1} defines a solution to the original problem \eqref{optimalvaluefct}. \\
	If on the other hand $b_1\leq b_2$ and $\mathbf{x}\in\mathcal{D}_2$, then either the optimal strategy can be defined through the solution of \eqref{eq_SpecialStrategyOptimizationProblem2}, or the optimal strategy ensures that the controlled process enters $\mathcal{D}_1$ with positive probability, which brings us back to the first case. Hence, the maximum of the solutions to \eqref{eq_SpecialStrategyOptimizationProblem1} and \eqref{eq_SpecialStrategyOptimizationProblem2} is at minimum a good approximation to the solution of \eqref{optimalvaluefct} if $\mathbf{x}\in\mathcal{D}_2$. 
	\\
	However, \eqref{eq_SpecialStrategyOptimizationProblem1} and \eqref{eq_SpecialStrategyOptimizationProblem2} are hard to solve explicitly, as $\tau^{L_1}$ (and $\tau^{L_2}$) and the expressions inside the indicator functions are strongly correlated, since all depend on the path of the underlying claim process $S(t)$. An approach to find approximate solutions to the problems in another subclass of the admissible strategies and using Monte-Carlo simulations is presented in the following section.

	\section{Approximation approach and simulation study} \label{Section_ApproximationSimulation}
	
	Given the theoretical results of the preceding section, we want to approximately solve the problems \eqref{eq_SpecialStrategyOptimizationProblem1} and \eqref{eq_SpecialStrategyOptimizationProblem2}. As the approach is identical, we are going to discuss only the former case. \\ 
	Assume that $x_2\leq 0$, while $x_1\geq 0$. We define the random process 

	\begin{equation*}
		\Lambda(x_2,s) := c_2 \cdot \mathds{1}_{\{\overline{X_2}(s)\geq 0\}} \cdot \mathds{1}_{\{\overline{X}_2(s)= X_2(s)\}}, \quad s\geq 0, 
	\end{equation*} such that we have (see \eqref{eq_SpecialStrategyValueFunc})

	\begin{align*}
		V^{\mathbf{L}^{*1}}(x_1,x_2)&=  \mathbb{E}_{x_1,x_2}\left[\int_0^{\tau^{L_1}} e^{-qs} \diff L_1(s) + \int_0^{\tau^{L_1}} \Lambda(x_2,s) \cdot e^{-qs} \diff s \right].
	\end{align*}
	The function $\Lambda$ depends explicitly on the initial value $x_2$ and the time $s$ and implicitly on the random path of the claim process $S$. Let $\omega\in\Omega$. Then $\Lambda(x_2,s)(\omega)$ is monotonically increasing with respect to $x_2$ and for any $s\geq 0$ it holds that 

	\begin{equation*}
		c_2 \geq \Lambda(x_2,s)	 \overset{x_2\to-\infty}{\longrightarrow} 0
	\end{equation*} almost surely.
	Equation \eqref{eq_SpecialStrategyOptimizationProblem1} may be reformulated as

	\begin{equation}
		\label{eq_SubproblemOptimization}
		\tilde{V}^{*1}(x_1,x_2) = \sup_{L_1\in\Pi_{x_1}} 	V^{\mathbf{L}^{*1}}(x_1,x_2). 
	\end{equation}
	As mentioned before, due to the complicated dependency structure between $\Lambda$ and $X_1(t)$, this problem seems impossible to solve exactly and explicitly. However, similar problems have been treated before, e.g., in \cite{Thonhauser2007}, where the authors consider the problem \eqref{eq_SubproblemOptimization} with $\Lambda(x_2,s)(\omega)\equiv\Lambda$ constant, i.e., (see \cite[Eq. (2)]{Thonhauser2007}) a problem of type

	\begin{equation} \label{eq_ThonhauserProblem}
		\tilde{V}_1(x_1) := \sup_{L_1\in\Pi_{x_1}} 	  \mathbb{E}_{x_1}\left[\int_0^{\tau^{L_1}} e^{-qs} \diff L_1(s) + \int_0^{\tau^{L_1}} \Lambda \cdot e^{-qs} \diff s \right].
	\end{equation}
	It is shown, that in the case of unbounded dividend payments, the corresponding HJB equation is given by (cf. \cite[Eq. 36]{Thonhauser2007})

	\begin{equation*}
		\max\left\{\Lambda + c\tilde{V}_1'(x_1) + \lambda \int_0^x \tilde{V}_1(x-y)\diff F_1(y) - (q+\lambda)\tilde{V}_1(x), 1-\tilde{V}_1'(x) \right\}= 0,
	\end{equation*} where $F_1$ is the cdf of the claims affecting branch one, i.e., $F_1(y) = F(y/b_1)$. Moreover, it turns out that in the case of exponential claims, the optimal strategy is a barrier strategy, cf. \cite[Prop. 11]{Thonhauser2007}. \\ 
	In the remainder of this section we will therefore restrict to barrier strategies as well. We say that an admissible strategy $L_1$ is of \emph{barrier type}, if there exists a fixed surplus level $x_1^b$, such that 
	
	\begin{itemize}
    \item $L_1$ does not pay any dividends if $X_1^{L_1}(t)< x_1^b$, 
    \item $L_1$ pays continuously $c_1\diff t$ as dividends, while $X_1^{L_1}(t) = x_1^b$,
    \item $L_1$ pays a lump sum of size $X_1^{L_1}(t) - x_1^b$ as dividends, if $X_1^{L_1}(t) > x_1^b$.
	\end{itemize}
	The set of all barrier strategies acting on branch one with initial capital $x_1$ is denoted by $\Pi_{x_1}^b$. We consider

	\begin{equation}
		\label{eq_SubproblemOptimization_Modified}
		\tilde{V}_1^b(x_1,x_2) := \sup_{L_1\in\Pi_{x_1}^b} 	\mathbb{E}_{x_1,x_2}\left[\int_0^{\tau^{L_1}} e^{-qs} \diff L_1(s) + \int_0^{\tau^{L_1}} \Lambda(x_2,s) \cdot e^{-qs} \diff s \right],
	\end{equation}
	which is a slight modification of \eqref{eq_SpecialStrategyOptimizationProblem1} (see also \eqref{eq_SubproblemOptimization}) and solve it using Monte-Carlo techniques. \\ 
	In order to apply the results from \cite{Thonhauser2007}, we assume in the following that the claims $U_i$ of the driving compound Poisson process $S(t)$ are exponentially distributed. Note that for modeling claim sizes there are more realistic distributions, such as log-normal or the generalized Pareto distribution. However, the exponential distribution is a popular choice for claim sizes in Cramér-Lundberg risk models, as this case is particularly well-treatable since e.g., ruin probabilities and excess of loss distributions at the time of ruin can be expressed in closed-form, see \cite{Asmussen2010}.
	In particular, let $U_i\sim\operatorname{Exp}(\gamma)$, $\gamma>0$, such that
	 
	 \begin{equation*}
		b_j U_i \sim \operatorname{Exp}\left(\gamma/b_j\right)	, \qquad j=1,2.
	\end{equation*}
	By \cite[Prop. 11]{Thonhauser2007} the barrier $x_1^*$ corresponding to the solution of \eqref{eq_ThonhauserProblem} fulfills 
	
	\begin{equation} \label{eq_ThonhauserBarrier}
		x_1^* := \frac{1}{R_1-R_2} \cdot \log\left(\frac{R_2^2 \cdot (\gamma/b_1 + R_2) \cdot (\gamma/b_1\cdot q + e^{R_1x_1^*}(\gamma/b_1 + R_1)\cdot \Lambda \cdot R_1)}{R_1^2 \cdot (\gamma/b_1 + R_1) \cdot (\gamma/b_1\cdot q + e^{R_2x_1^*}(\gamma/b_1 + R_2)\cdot \Lambda \cdot R_2)} \right),
	\end{equation} where $R_1,R_2$ are the roots of the polynomial

	\begin{equation} \label{eq_ThonhauserRoots}
		P(R) = c_1 R^2 + \left(\frac{\gamma}{b_1}\cdot c_1 - (q+\lambda)\right) \cdot R - \frac{\gamma}{b_1}\cdot q
	\end{equation} such that $R_2<0<R_1$. Hence, as $0\leq \Lambda(x_2,s)\leq c_2$, the barrier $x_1^b$ corresponding to the solution of \eqref{eq_SubproblemOptimization_Modified} is likely to be contained in the set of solutions of \eqref{eq_ThonhauserBarrier} for $\Lambda\in [0,c_2]$. \\
	For our numerical study, we consider model \eqref{eq_DefinitionModel} with parameters $\gamma= 0.25$, $\lambda =1$, $c_1 = 2,$ $c_2 = 4$, $b_1 = 0.25$, $b_2 = 0.75$ and $q = 0.05$ such that in particular assumption \eqref{eq_Assumption_c1b1c2b2} is fulfilled. Solving \eqref{eq_ThonhauserBarrier} for these parameters and $\Lambda\in[0,c_2]$ yields $x_1^b \in [7.00464,10.7136]$. Similar to the previous reasoning, we can consider $\mathbf{L}^{*2}$ and approximate the corresponding optimal strategy in branch two by a barrier strategy. In this case we obtain for the optimal barrier $x_2^b \in [10.8148,23.7285]$. \\
	Figure \ref{fig_barrierPlot} shows the estimated discounted dividend values of strategy $\mathbf{L}^{*1}=(L_1^b,L_2^*)$ and $\mathbf{L}^{*2}=(L_1^*,L_2^b)$ with respect to different barriers $x_1^b, ~x_2^b$ and for an initial capital of $\mathbf{x}= (25,25)$, such that we always start by lump-sum payments in both branches. As expected, in both cases the optimal barrier for our model is strictly in between the optimal barriers for problem \eqref{eq_ThonhauserProblem} with $\Lambda = 0$ and $\Lambda = c_2$ ($c_1$ for the case of $\mathbf{L}^{*2}$).
	
	\begin{figure}[H]
	\begin{adjustwidth}{-\extralength}{0cm}
        \centering
        \includegraphics[width=8.5cm]{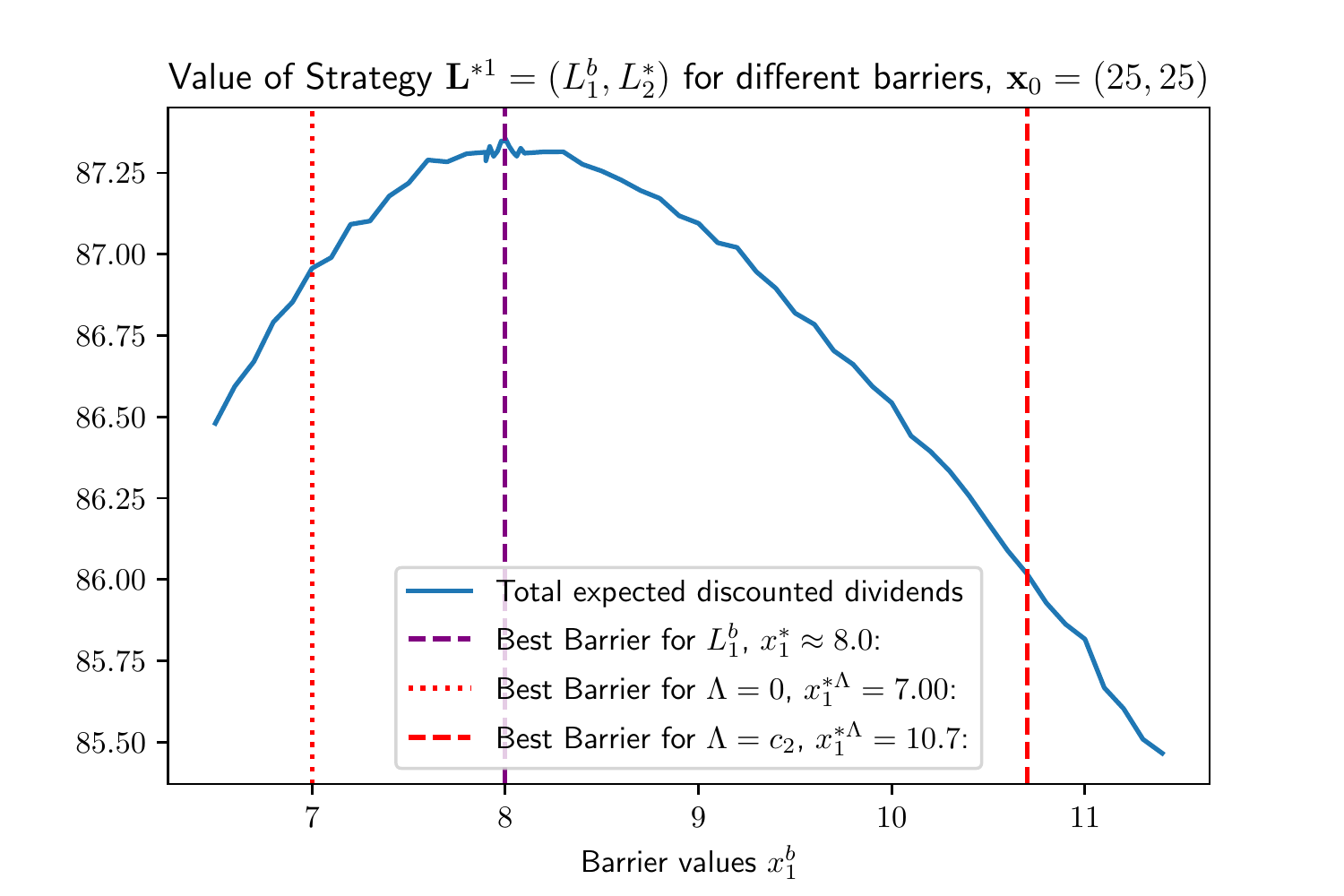}
        \includegraphics[width=8.5cm]{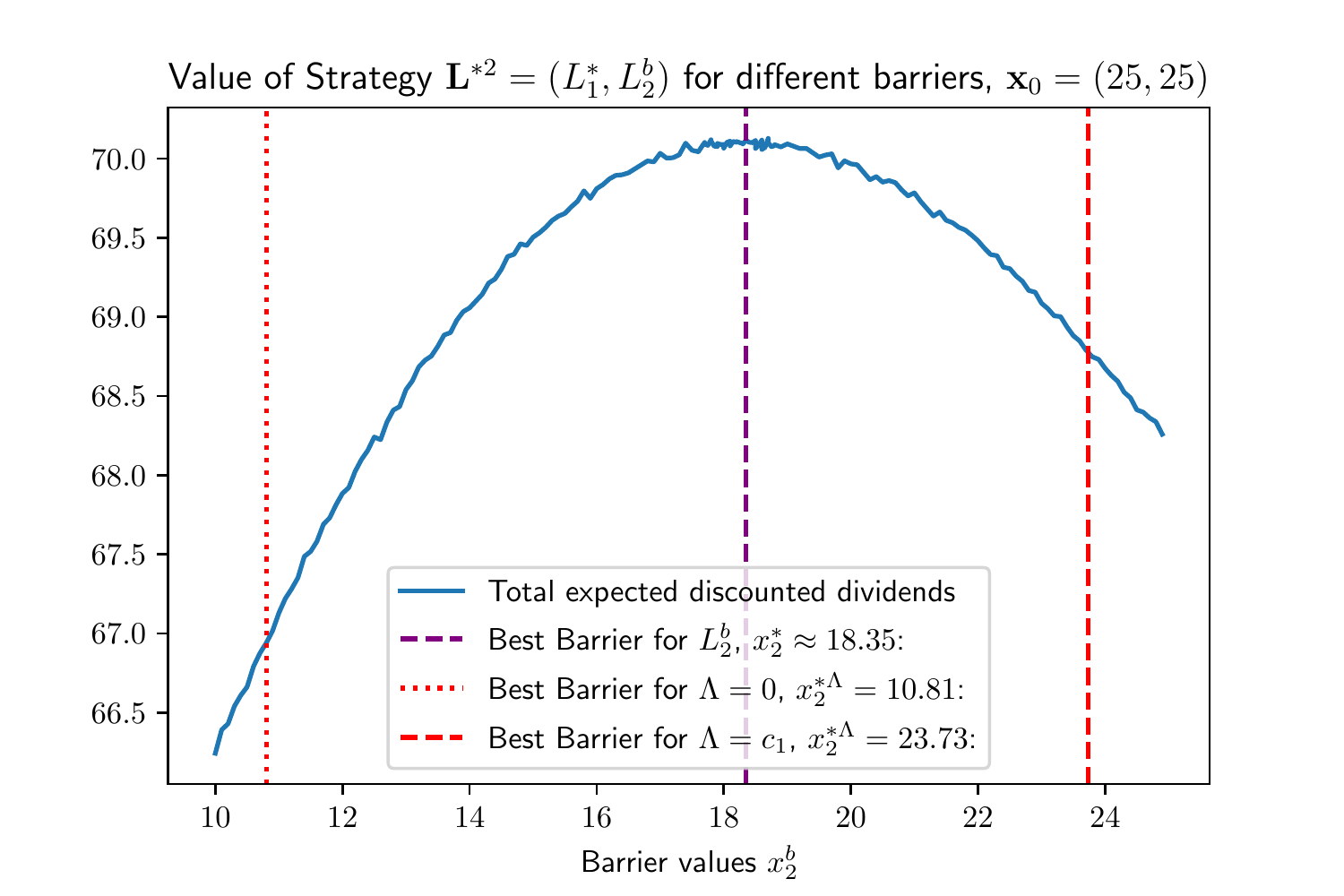}
        \end{adjustwidth}

		\caption{Expected discounted dividend payments of strategy $\mathbf{L}^{*1}=(L_1^b,L_2^*)$ and $\mathbf{L}^{*2}=(L_1^*,L_2^b)$ in dependence of the barriers $x_1^b, ~x_2^b$ chosen for strategies $L_1^b, ~L_2^b$.}
		\label{fig_barrierPlot}
	\end{figure}
	
	\begin{figure}[H]
	\begin{adjustwidth}{-\extralength}{0cm}
	\centering
		\includegraphics[width=16cm]{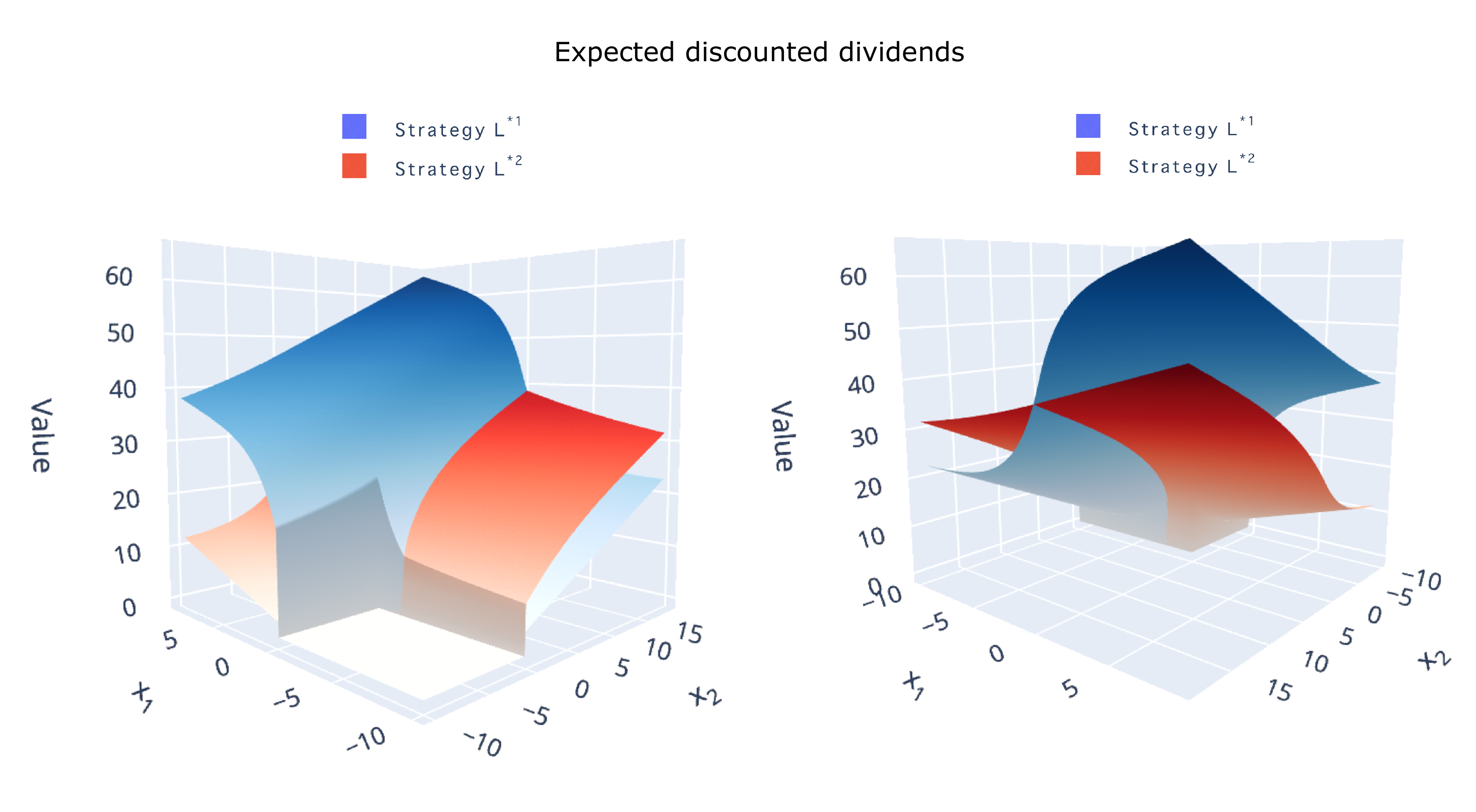}
	\end{adjustwidth}
		\caption{Approximated value functions $V^{\mathbf{L}^{*1}},~V^{\mathbf{L}^{*2}}$ with the optimal barriers $x_1^b=8.0, ~x_2^b = 18.35 $ with respect to different initial values $\mathbf{x}_0$ from several perspectives.}
		\label{fig_gridPlot}
	\end{figure}
	
	Figure \ref{fig_gridPlot} clearly shows that for our example $\mathbf{L}^{*1}$ is the preferable choice over $\mathbf{L}^{*2}$ for all initial values $\mathbf{x}\in\mathbb{R}^2_{\geq 0}$. Only on a subset of $(-\infty,0)\times(0,\infty)$ it yields a smaller expected value for the dividends, as already predicted in Lemma \ref{Lemma_SpecialStrategyNeverGoSet}. However, in general, a global statement that $\mathbf{L}^{*1}$ is better than $\mathbf{L}^{*2}$ (or vice versa) is not true, as our last example indicates:
	\begin{example}
		Consider a model, where $b_1 =b_2, ~c_1=c_2$. Then obviously claims happen along the line $x_1=x_2$ and, if $X_1(0)=X_2(0)$, we have $X_1(t) = X_2(t)$ for all $t\geq 0$. By construction, we have $V^{*1}(x,x) = V^{*2}(x,x)$ for all $x\geq 0$ and consequently Lemma \ref{Lemma_V*1vsV*2} implies, that $V^{*1}(\mathbf{x})\geq V^{*2}(\mathbf{x})$ on $\mathcal{D}_1$, while $V^{*2}(\mathbf{x})\geq V^{*1}(\mathbf{x})$ on $\mathcal{D}_2$. This shows that a general statement on whether $\mathbf{L}^{*1}$ or $\mathbf{L}^{*2}$ is better for all $\mathbf{x}\in\mathbb{R}^2_{\geq 0}$ can not hold in general. 
	\end{example}

\vspace{6pt}

\acknowledgments{We gratefully acknowledge the support of our Ph.D. supervisor Anita Behme and her valuable comments on earlier versions of the present manuscript. \\
Moreover, we thank three referees for suggestions and comments that helped improving this manuscript. \\ 
	We acknowledge the GWK support for funding this project by providing computing time through the Center for Information Services and HPC (ZIH) at TU Dresden. \\
}

\conflictsofinterest{The authors declare no conflict of interest.} 



%


\begin{adjustwidth}{-\extralength}{0cm}

\reftitle{References}


\bibliography{bibliography.bib}

%


\end{adjustwidth}
\end{document}